\documentclass[12pt]{amsart}

\usepackage[colorlinks]{hyperref}
\usepackage{color,graphicx,shortvrb}
\usepackage[latin 1]{inputenc}

\usepackage[active]{srcltx} 

\usepackage{enumerate}
\usepackage{amssymb}

\newtheorem{theorem}{Theorem}[section]

\newtheorem{corollary}[theorem]{Corollary}

\newtheorem{lemma}[theorem]{Lemma}

\newtheorem{problem}[theorem]{Problem}
\newtheorem{proposition}[theorem]{Proposition}

\DeclareGraphicsExtensions{.jpg,.pdf,.png,.eps}

\def\RR{{\mathbb{R}}}
\def\NN{{\mathbb{N}}}
\def\11{\textbf{$1$}}
\def\CC{{\mathbb{C}}}

\def\KK{{\mathbb{K}}}

\def\eqref{\emph}
\DeclareGraphicsExtensions{.jpg,.pdf,.png,.eps}

\begin{document}

\title[The Mazur--Ulam property on $\ell_\infty$-sum  and $c_0$-sum ]{The Mazur--Ulam property in $\ell_\infty$-sum  and $c_0$-sum of   strictly convex Banach spaces}

\author[J. Becerra Guerrero]{Julio Becerra Guerrero}

\address[J. Becerra Guerrero]{Departamento de An{\'a}lisis Matem{\'a}tico, Facultad de
Ciencias, Universidad de Granada, 18071 Granada, Spain.}
\email{juliobg@ugr.es}




\date{}

\begin{abstract} 	In this paper we deal with those Banach spaces $Z$ which satisfy the Mazur--Ulam property, namely that every surjective isometry $\Delta$ from the unit sphere of $Z$ to the unit sphere of any Banach space $Y$ admits an unique extension to a surjective real-linear isometry from $Z$ to $Y$. We prove that for every countable set $\Gamma$ with $\vert \Gamma \vert \geq 2$,  the Banach space $\bigoplus_{\gamma \in \Gamma}^{c_0} X_\gamma $ satisfies the Mazur--Ulam property, whenever the Banach space $X_\gamma $  is  strictly convex with dim$((X_\gamma )_{\RR})\geq 2$ for every $\gamma $. Moreover we prove that  the Banach space $C_0(K,X)$ satisfies the Mazur--Ulam property whenever $K$ is a  totally disconnected locally compact Hausdorff space  with  $\vert K\vert \geq 2$, and  $X$ is a strictly convex separable Banach space with dim$(X_{\RR})\geq 2$.  As consequences, we obtain the following results: (1) Every weakly countably determined Banach space can be equivalently renormed so that it satisfies the Mazur--Ulam property. (2) If $X$ is a strictly convex Banach space with dim$(X_{\RR}) \geq 2$,   then  $C(\mathfrak{C} ,X)$ satisfies the Mazur--Ulam property, where $ \mathfrak{C}$ denotes the Cantor set.
\end{abstract}

\maketitle
\thispagestyle{empty}

\section{Introduction}
Let $X$ and $Y$ be two Banach spaces over $\KK$ ($\RR$ or $\CC$) with unit spheres $S_X$ and $S_Y$, respectively. Then the clasical Mazur--Ulam theorem states that  every surjective isometry  $\Delta :X\to Y$ is  affine. As usual, by a \emph{convex body} of a normed space $X$ we mean a closed convex subset of $X$ with non-empty interior in $X$.  In 1972, Mankiewicz \cite{Mank1972} proved that every surjective isometry between convex bodies in two arbitrary normed spaces can be uniquely extended to an affine function between the spaces. Motivated by this observation, Tingley \cite{Ting1987} raised the following extension problem:

\begin{problem}\label{problem-Tingley} Suppose that $\Delta :S_X\to S_Y$ is a surjective isometry. Is $\Delta$ neccesarily the restriction of a surjective real-linear isometry from $X$ to $Y$?

\end{problem}

This problem has been addressed in many papers, and  has been answered affirmatively  for particular choices of $X$ and $Y$ (see \cite{Ding2002, Di:p, Di:C, Di:8, Ding07,  FerPe17b, FerPe17c, Liu2007, Ta:8, Tan2016, Wang}).

Let us say that a Banach space $X$ satisfies the \emph{Mazur--Ulam property} if, for any Banach space $Y$, every surjective isometry $\Delta: S_X\to S_Y$ admits an extension to a surjective real-linear isometry from $X$ to $Y$. The pioneering paper dealing with this property is that of Ding \cite{Ding07}, who proves that the space $c_0(\NN, \RR)$ of all null sequences of real numbers satisfies the Mazur-Ulam property. More examples of Banach spaces satisfying the Mazur--Ulam property are  $c(\Gamma ,\KK)$, $c_0(\Gamma , \KK)$, and $\ell_\infty (\Gamma , \KK)$ for any set $\Gamma $ (see \cite{Liu2007} for $\KK =\RR $ and \cite{JVMorPeRa2017, APe2017} for $\KK =\CC $).  More recently, it has been shown that this property is satisfied by unital complex C$^*$-algebras and real von Neumann algebras \cite{MoriOza2018},  by JBW$^*$-triples with rank one or rank bigger than or equal to three \cite{BeCuFePe}, and by the space $C(K, H)$ of all continuous functions from any compact Hausdorff space $K$ to a real or complex Hilbert space $H$ with $\dim (H_ \RR )\geq 2$ \cite{CuePer2}.  In \cite[Theorem 4.6]{JZLi}, Li  shows that, if $X_1$ and $X_2$ are strictly convex Banach spaces, then a surjective isometry $\Delta$ from the unit sphere of $X_1\bigoplus _\infty X_2$ to the unit sphere of any Banach space $Y$ admits an extension to a surjective real-linear isometry from $X_1\bigoplus _\infty X_2$ to $Y$ whenever  $\RR \Delta (S_{X_1})$ and $\RR \Delta (S_{X_2})$ are subspace of $Y$. Other references dealing with the Mazur--Ulam property are \cite{Pe2018, Ta:8, THL, TL1}. Anyway, as a matter of fact, Problem \ref{problem-Tingley} remains unanswered even if the Banach spaces $X$ and $Y$ are  two-dimensional  \cite{Ca-Sa, KadMar2012, Wang-Hua}.

The main aim of this paper is to provide the reader with new examples of Banach spaces satisfying the Mazur--Ulam property.

In Section 2, we revisit some previously known results, like \cite{BoMeNav}, \cite{JMNS}, \cite[Corollary 9]{MeNav1995}, and \cite[Theorem 5]{Peck1968}, in order to establish the following:
\begin{enumerate} \item If $\{X_\gamma \}_{\gamma \in \Gamma}$ is a family of strictly convex Banach spaces with dim$((X_\gamma )_{\RR}) \geq 2$ and $\vert \Gamma \vert \geq 2$ (where $\vert \cdot \vert $ means cardinality), then the closed unit ball of  $\bigoplus_{\gamma \in \Gamma}^{\ell_\infty } X_\gamma $ is the convex hull of its extreme points (Proposition \ref{convex-hull-extreme-point}).
\item If $X$ is a strictly convex Banach space with dim$(X_{\RR}) \geq 2$, and if $K$ is a totally disconnected  compact Hausdorff space with $\vert K\vert \geq 2$, the the closed unit ball of $C(K,X)$  is the convex hull of its extreme points  (Proposition  \ref{C(K,X)-Strong Mankiewicz property}).
\end{enumerate}

Section 3 is devoted to proving our results on the Mazur--Ulam property in $c_0$- and $\ell_\infty$-sums of  families of Banach spaces. We show that, if $\{X_\gamma \}_{\gamma \in \Gamma}$ is any family of strictly convex Banach spaces such that  $\vert \Gamma \vert \geq 2$, and such that dim$((X_\gamma )_{\RR})\geq 2 $ and the norm of $X_\gamma$  is G\^{a}teaux differentiable in a dense subset of its unit sphere for every $\gamma $, then both $\bigoplus_{\gamma \in \Gamma}^{c_0} X_\gamma $ and $\bigoplus_{\gamma \in \Gamma}^{\ell_\infty } X_\gamma $ satisfy the Mazur--Ulam property  (Theorem \ref{suma-Mazur-Ulam}). In the particular cases of  countable  $c_0$-sums and finite $\ell_\infty$-sums, the hypothesis of G\^{a}teaux differentiability of the norm can be removed. Indeed, we show that for  a family $\{X_n \}_{n\in \NN}$ of strictly convex Banach spaces with dim$(X_n)\geq 2$, the Banach spaces $X_{1} \oplus _\infty \cdots \oplus _\infty X_{n} $ ($n \geq 2$)  and $\bigoplus_{n \in \NN }^{c_0} X_n $ satisfies the Mazur--Ulam property (Theorem \ref{main}). As a direct consequence, we obtain that every strictly convex Banach space  can be equivalently renormed so that it satisfies the Mazur--Ulam property  (Corollary \ref{renorm}). For instance, every separable Banach space, every reflexive Banach space, and  more generally every weakly Lindelof Banach space (see \cite{ArMe} and \cite{DevGodZiz93}), can be equivalently renormed so that it satisfies the Mazur--Ulam property. A characterization in linear topological terms of the normed spaces which are strictly convex renormable  can be found in  \cite{MoOrTrZi}.

The concluding Section 4 is devoted to studying the Mazur--Ulam property  in the Banach space $C_0(K,X)$ of all continuous functions  vanishing at infinity from a locally compact Hausdorff space $K$ to a Banach space $X$.  As the most outstanding result in this setting, we show that, if $K$ is totally disconnected  with $\vert K\vert \geq 2$, if $X$ is  strictly convex with dim$(X_{\RR}) \geq 2$,  and if the norm of $X$  is G\^{a}teaux differentiable in a dense subset of its unit sphere, then $C_0(K,X)$ satisfies the Mazur--Ulam property (Theorem \ref{C(K)-disconex-non-smooth-localy}). In the case that $K$ is actually a metrizable compact space, the hypothesis of G\^{a}teaux differentiability of the norm can be removed. As a consequence, if $X$ is a strictly convex Banach space with dim$(X_{\RR}) \geq 2$,   then  $C(\mathfrak{C} ,X)$ satisfies the Mazur--Ulam property, where $ \mathfrak{C}$ denotes the Cantor set (Corollary \ref{Cantor}).

\smallskip

{\bf Notation.}
Given a  Banach space $X$, $B_X$, $S_X$, and $X^*$ shall stand for
the closed unit ball, the unit sphere, and the dual of $X$, respectively. We denote by $Ext(B_X)$ the set of all  extreme points of $B_X$.

Given a family $\{X_\gamma \}_{ \gamma \in \Gamma }$ of Banach spaces, we set  $Z_0:=\bigoplus_{\gamma \in \Gamma}^{c_0} X_\gamma $ and $Z_\infty :=\bigoplus_{\gamma \in \Gamma}^{\ell_{\infty}} X_\gamma $.  Given a subset $\mathcal R\subseteq \Gamma $, we denote by $P_{\mathcal R}$ the canonical projection from $\bigoplus_{\gamma \in \Gamma } X_\gamma $ to $\bigoplus_{\gamma \in \mathcal{R}} X_\gamma $. The symbol $Z$ shall stand for any of the spaces  $Z_0$ or $Z_\infty$.

We recall that  a Banach space $X$ is said to be \emph{strictly convex} if every element of $S_X$ is an extreme point of $B_X$.

Given a compact Hausdorff space  $K$, the symbol $C(K,X)$ shall stand for the Banach space of all continuous functions from $K$ to $X$ equipped with the sup norm. Given $f\in C(K)$ and $x\in X$, we denote by $f\otimes x$ the function in $C(K,X)$ defined by $(f\otimes x)(t):=f(t)x$ for all $t\in K$. If $K$ is a locally compact Hausdorff space, we denote by $C_0(K,X)$ the space of continuous $X$-valued functions on $K$ vanishing at infinity. Recall that $f:K\to X$ vanishes at infinity if for every $\varepsilon >0$ there exists a compact subset $K_\varepsilon$ of $K$ satisfying $\Vert f(t)\Vert <\varepsilon $  for all $t\in K\setminus K_\varepsilon$.

It is clear that $X$ satisfies the Mazur--Ulam property if and only if so does $X_{\RR}$. According to this remark, throughout this paper we shall assume that all Banach spaces are real.

\smallskip

\section{The strong Mankiewicz property: preliminary results}

The result of  Mankiewicz in \cite{Mank1972} is one of the main tools applied in those papers devoted to explore new progress on Tingley's problem and to determine new Banach spaces satisfying the Mazur--Ulam property.\smallskip

In the recent paper \cite{MoriOza2018}, Mori and Ozawa introduce new techniques that are essential for our work.  Following these authors, we shall say that a convex subset $C$ of a normed space $X$ satisfies the \emph{strong Mankiewicz property} if every surjective isometry $\Delta$ from $C$ to an arbitrary convex subset $L$ in a normed space $Y$ is affine. In \cite[Theorem 2]{MoriOza2018} it is show that some of the hypothesis in Mankiewicz's theorem can be somehow relaxed. The precise result reads as follows.

\begin{theorem}\label{Mori-Ozawa}\cite[Theorem 2]{MoriOza2018} Let $X$ be a Banach space such that the closed convex hull of $Ext(B_X)$  has non-empty interior in $X$. Then, every convex body $K\subseteq X$ has the strong Mankiewicz property. $\hfill\Box$
\end{theorem}

Throughout this section we shall work with  a family of nonzero Banach spaces $\{X_\gamma \}_{ \gamma \in \Gamma }$.

It is known that $p\in Ext(B_{Z_\infty})$ implies $\|p(\gamma )\|_{ X_\gamma } = 1$ for all $\gamma \in \Gamma$. In the case that $X_\gamma$ is strictly convex for every $\gamma \in \Gamma$, the converse implication is  true. Therefore we are provided with the following.

\begin{lemma}\label{extreme points} Suppose that $X_\gamma$ is strictly convex for every $\gamma \in \Gamma$.  Then $p\in Ext (B_{Z_\infty})$ if and only if $\|p(\gamma )\|_{X_\gamma } = 1$ for all $\gamma  \in \Gamma$.
\end{lemma}

The following lemma is folklore.

\begin{lemma}
Let $X$ be a Banach space with $\dim (X)\geq 2$. Then every element in $B_X$ can be expressed as a mean of two elements in $S_X$.
\end{lemma}
\begin{proof}
Let $x$ be in $X$ with $0<\Vert x\Vert <1$. Then
$$\mbox{$\left\Vert 2x-\frac{x}{\Vert x\Vert}\right\Vert <1$ \ and \ $\left\Vert 2x+\frac{x}{\Vert x\Vert}\right\Vert >1$.}$$
Therefore, since $S_X$ is connected, there exists $y\in S_X$ such that \linebreak $\Vert 2x-y\Vert =1$. Now $x=\frac{1}{2}(2x-y+y)$.
 \end{proof}

As a straightforward consequence of the above lemma, we derive the following.

\begin{corollary} \label{Mena}
Suppose that $\dim (X_\gamma )\geq 2$ for every $\gamma $. Then every  element in $z\in B_{Z_\infty}$ can be written as $z=\frac{1}{2}(x+y)$ with  $\|x(\gamma )\|_{X_\gamma } = 1=\|y(\gamma )\|_{X_\gamma }$ for all $\gamma  \in \Gamma$.
\end{corollary}

Now, combining Lemma \ref{extreme points} and Corollary \ref{Mena}, we obtain the following.


\begin{proposition}\label{convex-hull-extreme-point} Suppose that $\dim (X_\gamma )\geq 2$ and that $X_\gamma$ is strictly convex for every $\gamma \in \Gamma$. Then every  element in $B_{Z_\infty}$ admits a  expression as a mean of two elements in $Ext (B_{Z_\infty})$.
\end{proposition}

\begin{proposition}\label{Strong Mankiewicz property} Suppose that $\dim (X_\gamma )\geq 2$ and that $X_\gamma$ is strictly convex for every $\gamma \in \Gamma$. Then every convex body in $Z_\infty$ satisfies the strong  Mankiewicz property.
	
\end{proposition}
\begin{proof} By Proposition \ref{convex-hull-extreme-point}, the closed unit ball of $Z_\infty$ is the convex hull of its extreme points, and hence the result follows from the Mori--Ozawa Theorem \ref{Mori-Ozawa}.
\end{proof}

For each $\gamma _0\in \Gamma$ and each $x_{\gamma _0}\in S_{X_{\gamma _0 }}$ we set $$A(\gamma _0,x_{\gamma _0}):=\{ z\in S_{Z} : z(\gamma _0) = x_{\gamma _0}\}.$$

\begin{lemma}\label{maximal}  Suppose that  $X_\gamma$ is strictly convex for every $\gamma \in \Gamma$. Then, for $\gamma _0\in \Gamma$ and  $x_{\gamma _0}\in S_{X_{\gamma _0 }}$,  $A(\gamma _0,x_{\gamma _0})$ is a  maximal norm-closed proper face of $B_{Z}$, equivalently, a maximal convex subset of $S_{Z}$.
\end{lemma}
\begin{proof} 	We define $z_0\in S_{Z}$  by $z_0(\gamma ):=0$ if $\gamma \neq \gamma _0$ and $z_0(\gamma _0):=x_{\gamma _0}$. Since $A(\gamma _0,x_{\gamma _0})$ is a convex subset of $S_{Z}$, by Zorn's lemma, there exists a maximal convex subset $C$ of $S_{Z}$ that contains $A(\gamma _0,x_{\gamma _0})$. By \cite[Lemma 3.1]{Tan2016preprint}, there exists $z^*\in S_{Z^*}$ such that $C=\{z\in B_{Z}:z^*(z)=1\}$. Given $z\in C$ we have that $z^*(z+z_0)=2=\Vert z+z_0\Vert $. This implies that $\Vert z(\gamma _0)+x_{\gamma _0}\Vert _{X_{\gamma _0}}=2$, and hence $z(\gamma _0)=x_{\gamma _0}$ because $X_{\gamma_0}$ is strictly convex. We conclude that $C=A(\gamma _0,x_{\gamma _0})$.\end{proof}


 The proof of the next result is an adaptation of  \cite[Lemma 2.1]{JVMorPeRa2017}.

\begin{lemma}\label{l existence of support functionals for the image of a face} Suppose that  $X_\gamma$ is strictly convex for every $\gamma \in \Gamma$. Let $Y$ be a  Banach space, and let $\Delta : S_Z\to S_Y$ be a surjective isometry. Then, for  $\gamma _0\in \Gamma$ and  $x_{\gamma _0}\in S_{X_{\gamma _0}}$, the set $${\rm supp}(\gamma _0,x_{\gamma _0}) := \{\psi\in S_{Y^*} : \  \psi^{-1} (\{1\})\cap B_{Y} = \Delta(A(\gamma _0,x_{\gamma _0})) \}$$ is a non-empty weak$^*$-closed face of $B_{Y^*}$.
\end{lemma}
\begin{proof} By Lemma \ref{maximal}, the set $A(\gamma _0,x_{\gamma _0})$ is a maximal convex subset of $B_Z$. Then, by \cite[Lemma 3.5]{Tan2014}, $\Delta (A(\gamma _0,x_{\gamma _0}))$ is a maximal convex subset of $B_Y$, and  therefore, by \cite[Lemma 3.1]{Tan2016preprint}, there exists $\psi \in S_{Y^*}$ such that $\Delta (A(\gamma _0,x_{\gamma _0}))=\{y\in B_{Y}:\psi (y)=1 \}$. It is clear that given $\psi ' \in \overline{{\rm supp}(\gamma _0,x_{\gamma _0}) }^{w^*}$, we have $\Delta (A(\gamma _0,x_{\gamma _0}))\subseteq \{y\in B_{Y}:\psi ' (y)=1 \} $. Since  $\Delta (A(\gamma _0,x_{\gamma _0}))$ is a maximal convex subset of $B_Y$, $\psi '\in {\rm supp}(\gamma _0,x_{\gamma _0})$.
\end{proof}

Given a norm-one element $x$ in a Banach space $X$, the star of $x$ with respect to $S_X$, $St(x)$, is defined by
$$St(x):=\{ y\in S_X:\Vert x+y\Vert =2 \}.$$ It is known that  $St(x)$ is precisely the union of all maximal convex subsets of $S_X$ containing $x$. The arguments in the proofs of \cite[Lemmas 2.2 and 2.3, and Proposition 2.4]{JVMorPeRa2017}  allow us  to obtain the following result.

\begin{lemma}\label{l tech 2} Suppose that  $X_\gamma$ is strictly convex for every $\gamma \in \Gamma$. Let $Y$ be a Banach space and let $\Delta : S_{Z}\to S_Y$ be a surjective isometry. Pick $\gamma _0\in \Gamma $, $x_{\gamma _0}\in S_{X_{\gamma _0}}$ and define $z_0\in S_Z$  by $z_0(\gamma):=x_{\gamma _0}$ if $\gamma =\gamma_0$ and $z_0(\gamma):=0$ if $\gamma \neq \gamma_0$.   Then the following assertions hold:	
\begin{enumerate}
\item $St(\Delta (z_0))=\Delta (A(\gamma _0,x_{\gamma _0}))$.
\item $\psi \Delta(z) =-1$ for all $z\in A(\gamma _0,-x_{\gamma _0})$ and $\psi \in {\rm supp}(\gamma _0,x_{\gamma _0})$.
\item   $\Delta (-A(\gamma _0,x_{\gamma _0}))=-\Delta (A(\gamma _0,x_{\gamma _0}))$.
\item ${\rm supp}(\gamma _0,x_{\gamma _0})\cap {\rm supp}(\gamma ,x)=\varnothing$, for all  $\gamma \neq \gamma_0$ and $x\in S_{X_{\gamma }}$.
\item ${\rm supp}(\gamma _0,x_{\gamma _0})\cap {\rm supp}(\gamma _0,x')=\varnothing$, for every $ x'\in S_{X_{\gamma_0}}$ with \linebreak $x'\neq x_{\gamma _0}$.
\item  If  $\psi \in {\rm supp}(\gamma _0,x_{\gamma _0}),$ and  if $z\in S_{Z}$ with $z(\gamma _0)=0$, then $\psi \Delta(z) =0$. Furthermore, $\vert \psi (\Delta (z))\vert <1$, for all $\psi \in {\rm supp}(\gamma_0 ,x_0)$ and $z\in S_Z$ with $\Vert z(\gamma_0)\Vert <1$.
\item  If $z$ is in $S_Z$ such that $\psi \Delta(z) =0$ for all    $x\in S_{X_{\gamma_0 }}$ and $\psi \in {\rm supp}(\gamma_0 ,x)$, then $z(\gamma _0)  =0$.
\end{enumerate}
\end{lemma}

\medskip

By  \cite[Lemma 1.4]{AronLohm}, if $p\in Ext(B_{C(K,X)})$, then $\|p(t)\| = 1$  for all $t \in K$. The reciprocal implication is not true in general, however by \cite[Remark 1.5]{AronLohm}, if $X$ is a strictly convex Banach space, then  \begin{equation}\label{eq extreme points conts functions} \hbox{$p\in Ext (B_{C(K,X)})$ if and only if $\|p(t)\| = 1$ for all $t \in K$}.
\end{equation} \smallskip

Given a compact Hausdorff space $K$, we denote by dim$(K)$ the \emph{covering dimension} of $K$ \cite[page 385]{Engelking89}. We recall that a space $K$ has dim$(K)=0$ if and only if each point of $K$ has a neighborhoods base consisting of clopen sets \cite[Definition 29.4]{Willard}. We shall simply observe that, if $K$ is a totally disconnected compact Hausdorff  space, then  dim$(K)=0$ \cite[Theorem 29.7, page 211]{Willard}.

\begin{proposition}\label{C(K,X)-Strong Mankiewicz property} Let $X$ be a strictly convex Banach space and let $K$ be a compact Hausdorff space satisfying one of the following conditions:\begin{enumerate}[$(1)$]\item $X$ is infinite dimensional.
		\item  $X$ is $n$-dimensional with $n\geq 2$ {\rm(}$n\in \mathbb{N}${\rm)} and  dim$(K)\leq n-1$.
	\end{enumerate} Then $B_{C(K, X)}$ is the convex hull of its extreme points, and hence every convex body in $C(K,X)$ satisfies the strong Mankiewicz property.
\end{proposition}
\begin{proof} Under any one of the conditions (1) or (2), we deduce from \cite[Corollaries 8 and 9]{MeNav1995} that the closed unit ball of $C(K, X)$ is the convex hull of its extreme points. Therefore the desired conclusion is a consequence of  Mori--Ozawa's Theorem \ref{Mori-Ozawa}.
\end{proof}

In the remaining of this section, $K$ shall denote a locally compact Hausdorff space, and $X$ shall denote a   strictly convex Banach space.

For each $t _0\in K$ and each $x_0\in S_X$ we set $$A(t_0,x_0):=\{ f\in S_{C_0(K, X)} : f(t_0) = x_0\}.$$

\begin{lemma}\label{maximal-bis}  Let $t _0\in K$ and  $x_0\in S_X$. Then $A(t _0,x_0)$ is a  maximal norm-closed proper face of $B_{C_0(K, X)}$, equivalently, a maximal convex subset of $S_{C_0(K, X)}$.
\end{lemma}
\begin{proof} Let $U$ be an open set of $K$ with $t_0\in U$. Applying Urysohn's lemma we find $h\in C_0(K)$ with $0\leq h\leq 1$, $h(t_0)=1$ and $h\arrowvert _{K\setminus U}=0$. We consider $h\otimes x_0\in S_{C_0(K, X)} $. Since $A(t _0,x_0)$ is a convex set of $S_{C_0(K, X)}$, by Zorn's lemma, there exists a maximal convex subset $C$ of $S_{C_0(K, X)}$ that contains $A(t _0,x_0)$. By \cite[Lemma 3.1]{Tan2016preprint}, there exists $\varphi \in S_{C_0(K, X)^*}$ such that $C=\{f\in B_{C_0(K, X)}:\varphi (f)=1\}$. Given $g\in C$ we have that $\varphi (h\otimes x_0+g)=2=\Vert h\otimes x_0+g\Vert $. This implies that there exists $t_1\in K$ such that $\Vert h(t_1)x_0+g(t_1)\Vert =2$. Since $\Vert h(t_1)x_0\Vert =1=\Vert g(t_1)\Vert $ and  $0\leq h\leq 1$, we have that $h(t_1)=1$. Now the strict convexity of $X$ yields $x_0=g(t_1)$. We conclude that $C=A(t_0,x_0)$. \end{proof}
	

The proof of the next result is  similar to that of Lemma \ref{l existence of support functionals for the image of a face}.

\begin{lemma}\label{l existence of support functionals for the image of a face-bis} Let $Y$ be  a Banach space and let $\Delta : S_{C_0(K, X)}\to S_Y$ be a surjective isometry. Then for each $t_0\in K$ and each $x_0\in S_X$ the set
	$${\rm supp}(t_0,x _0) := \{\psi\in S_{Y^*} : \  \psi^{-1} (\{1\})\cap B_{Y} = \Delta (A(t_0,x_0)) \}$$
	is a non-empty weak$^*$-closed face of $B_{Y^*}$.$\hfill\Box$
\end{lemma}

With the appropriate changes, the proof of the next lemma follows the arguments in the proofs of \cite[Lemmas 2.4 and 2.5, and Propositions 2.6 and 3.1]{CuePer}.

\begin{lemma}\label{l tech 2-bis} Let $Y$ be a Banach space and let $\Delta : S_{C_0(K, X)}\to S_Y$ be a surjective isometry. Pick $t _0\in K $ and $x_{0}\in S_{X}$. Then the following assertions hold:	
	\begin{enumerate}
		\item $\psi \Delta(f) =-1$ for all $f\in A(t_0,-x_{0})$ and $\psi \in {\rm supp}(t _0,x_{ _0})$.
		\item   $\Delta (-A(t_0,x_{0}))=-\Delta (A(t _0,x_{0}))$.
		\item ${\rm supp}(t_0,x_{0})\cap {\rm supp}(t ,x)=\varnothing$, for all  $t \neq t_0$ and $x\in S_{X}$.
		\item ${\rm supp}(t _0,x_{0})\cap {\rm supp}(t_0,x')=\varnothing$, for every $x'\in S_{X}$ with $x'\neq x_{0}$.
		\item  If  $\psi \in {\rm supp}(t _0,x_{0}),$ and if $f\in S_{C_0(K, X)}$ with $f(t _0)=0$, then \linebreak $\psi \Delta(f) =0$. Furthermore, $\Vert \psi (\Delta (f))\Vert < 1$ for all $f\in S_{C_0(K, X)}$  with $\Vert f(t _0)\Vert <1$ and  $\psi \in {\rm supp}(t _0,x_{0})$.
		\item If $K$ is totally disconnected, then $f(t _0)  =0$ whenever $f$ is  in $S_{C_0(K, X)}$ with $\psi \Delta(f) =0$ for all    $x\in S_{X}$ and $\psi \in {\rm supp}(t_0 ,x)$.
	\end{enumerate}
\end{lemma}
\begin{proof} Let $O$ be an open set of $K$  with $t_0\in O$. Via Urysohn's lemma, we find $h\in C_0(K) $ with $0\leq h\leq 1$, $h(t_0)=1$ and $h|_{K\setminus O}=0$. Let us set $f_0=h\otimes x_0$.
	
First we see that $St(\Delta (f_0))= \Delta (A(t_0,x_{0}))$. Let $y\in \Delta (A(t_0,x_{0}))$. Then $\Delta ^{-1}(y)\in A(t_0,x_{0})$, and hence $\Vert \Delta ^{-1}(y)+f_0\Vert =2$. By \cite[Corollary 2.2]{FangWang06}, $\Vert y+\Delta (f_0)\Vert =2$, and so $y\in St(\Delta (f_0))$. Conversely, let $y\in St(\Delta (f_0))$. Then $\Vert y+\Delta (f_0)\Vert =2$ and hence, by \cite[Corollary 2.2]{FangWang06}, $\Vert \Delta ^{-1}(y)+f_0\Vert =2$. This implies that $\Vert \Delta ^{-1}(y)(t_0)+x_{0}\Vert =2$. Now, since $X$ is strictly convex,  $\Delta ^{-1}(y)(t_0)=x_{0}$, and so $\Delta ^{-1}(y)\in A(t_0,x_{0})$.

	\medskip

	(1) Let  $f\in A(t _0,-x_{0})$ and $\psi \in {\rm supp}(t _0,x_{0})$. Since $-\Delta (f)\in S_Y$, there exists $h\in S_{C_0(K,X)}$ with $\Delta (h)=-\Delta (f)$. Let $g\in A(t _0, x_{0})$. Then  $2=\Vert g-f\Vert =\Vert \Delta ( g)-\Delta (f)\Vert$, and hence  $\Delta (h)=-\Delta (f)\in St(\Delta (g)) $ and $\Vert \Delta ( g)+\Delta (h)\Vert =2$. By \cite[Corollary 2.2]{FangWang06}, $\Vert  g+ h\Vert =2$ for every $g\in A(t _0,x_{0})$. Then there exists $t_g\in K$ such that $g(t_g)=h(t_g)$. Let $\mathcal{F}$ be the family of all open subset of $K$ containing $t_0$, and let $O\in \mathcal{F}$. Then, by Urysohn's lemma, there exists  $g_O\in C_0(K) $ with $0\leq g_{O}\leq 1$, $g_{O}(t_0)=1$ and $g_{O}|_{K\setminus O}=0$. Since $g_O\otimes x_0\in A(t _0,x_{0})$, there exists $t_O\in K$ such that $g_O(t_O)x_0=h(t_O)$ and $\Vert g_O(t_O)x_0\Vert =\Vert h(t_O)\Vert =1$. Since $g_O(t_O)\in [0,1]$ and $\Vert g_O(t_O)x_0\Vert =1$, we have $g_O(t_O)=1$, hence $x_0=h(t_O)$. Considering $\mathcal{F}$ as a directed set under the reverse inclusion, the net $(t_O)_{O\in \mathcal{F}}$ converges to $t_0$. It follows from the continuity of $h$ that  $h(t_0)=x_0$ and $h\in A(t _0,x_{0})$.  Therefore $\psi (\Delta (h))=1=- \psi (\Delta (f))$.

	\medskip

	(2) By (1),  $\psi (\Delta (f))=-1$ whenever $f\in A(t _0,-x_{0})$ and $\psi \in {\rm supp}(t_0,x_{0})$. Since $A(t _0,-x_{0})=-A(t _0,x_{0})$ and $\psi ^{-1} (1)\cap B_{C_0(K, X)}=\Delta (A(t_0,x_{0}))$, we have that $$ \Delta (-A(t _0,x_{0}))=\psi ^{-1}(-1)\cap B_{C_0(K, X)}=- \Delta (A(t _0,x_{0}))$$  for every $\psi \in {\rm supp}(t_0,x_{0})$.

	\medskip
	
	(3) We argue by contradiction, and hence we assume that there exist $t_1 \neq t_0$, $x_1\in S_X$,  and $\psi \in {\rm supp}(t _0,x_0)\cap {\rm supp}(t_1,x_{1})$. By Urysohn's lemma, there are $u,v\in C_0(K) $ with $0\leq u,v\leq 1$, $u(t_0)=1=v(t_1)$ and $u(t) v(t)=0$ for every $t\in K$.  Define the elements $f_1:=u\otimes x_0\in A(t_0,x_{0})$ and $f_2:=v\otimes x_1\in A(t_1,x_{1})$. By (1), $\psi (\Delta (-f_2))=-1$, and hence $$2=\psi (\Delta (f_1))-\psi (\Delta (-f_2))\leq \Vert \Delta (f_1)- \Delta (-f_2)  \Vert =\Vert f_1+f_2  \Vert =1,$$
	the desired contradiction.
	
	\medskip
	
	(4) Let $u\in C_0(K)$ and $x'\in S_X$ be such that $u(t_0)=1$ and  $x_0\neq x'$. Then $f_0:=u\otimes x_0\in A(t_0,x_{0})$ and $f_1:=u\otimes x'\in A(t_0,x')$. If $\psi \in {\rm supp}(t _0,x_0)\cap {\rm supp}(t_0,x')$, then $$2=\psi (\Delta (f_0))+\psi (\Delta (f_1))\leq \Vert \Delta (f_0)+ \Delta (f_1)  \Vert =\Vert f_1+f_2  \Vert =\Vert x_0+x'\Vert <2 ,$$ a contradiction.
	Therefore ${\rm supp}(t _0,x_0)\cap {\rm supp}(t_0,x')=\varnothing $, as desired.
	
	\medskip

	(5) Let  $f$ be in $S_{C_0(K, X)}$, and let $\psi $ be in ${\rm supp}(t _0,x_{0})$. Suppose that $f(t_0)=0$. Let $0<\varepsilon <1$, and define the open set $$O_\varepsilon :=\{t\in K: \Vert f(t)\Vert  < \varepsilon \}.$$ By Urysohn's lemma, there exists $g\in C_0(K)$  with $0\leq g\leq 1$, $g(t_0)=1$ and $g(K\setminus O_\varepsilon)=0$. Then $\Vert f\pm g\otimes x_0\Vert \leq 1+\varepsilon$. By (2), $\psi (\Delta (-g\otimes x_0))=-1$ and, by definition, $\psi (\Delta (g\otimes x_0))=1$. It follows $$\vert \psi (\Delta (f))\pm 1\vert = \vert \psi (\Delta (f))- \psi (\Delta (\pm g\otimes x_0))\vert  \leq \Vert \Delta (f)- \Delta (\pm g\otimes x_0) \Vert  $$$$\leq\Vert f \pm g\otimes x_0 \Vert \leq 1+\varepsilon .$$
	By letting $\varepsilon \to 0$, we obtain $\vert \psi (\Delta (f))\pm 1\vert \leq 1$. This implies $\psi (\Delta (f))=0$.
	Now suppose merely that $\Vert f(t_0)\Vert <1$. Pick $0<\varepsilon <1$ with $\Vert f(t _0)\Vert <1 -\varepsilon$ and consider the closed set $$C_\varepsilon :=\{t\in K: \Vert f(t)\Vert  \geq 1-\varepsilon \}.$$ By Urysohn's lemma, there exists $h\in C_0(K)$ such that $0\leq h\leq 1$, $h(t_0)=0$ and $h(C_\varepsilon )=1$. The function $hf$ in $S_{C_0(K, X)}$ satisfies $(hf)(t_0)=0$. Consequently, $\psi (\Delta (hf))=0$. We conclude that $$\vert \psi (\Delta (f))\vert =\vert \psi (\Delta (f)) -\psi (\Delta (hf))\vert \leq \Vert \Delta (f) -\Delta (hf)\Vert  $$$$=\Vert f-hf\Vert \leq 1-\varepsilon <1. $$
	
	\medskip

	(6) Suppose that $K$ is totally disconnected. Let $f$ be in $S_{C_0(K, X)}$ such that  $\psi (\Delta (f))=0$ for all  $x\in S_{X}$ and $\psi \in {\rm supp}(t_0,x)$. Assume towards a contradiction that $f(t_0)\neq 0$. If $\Vert f(t_0)\Vert =1$, then $\psi (\Delta (f))=1$ whenever $\psi $ lies in ${\rm supp}(t_0,f(t_0))$, and this is impossible. Therefore  $\Vert f(t_0)\Vert <1$.  It is known that  a totally disconnected locally compact Hausdorff space has a basis of its  topology consisting of compact open sets. We can always find a compact open subset $\mathcal{O}$ satisfying  that $$t_0\in \mathcal{O}\subseteq \Big\{s\in K: \Vert f(s)-f(t_0)\Vert <\frac{\Vert f(t_0)\Vert }{2}  \Big\}.$$ Since $0<\frac{1}{2}\Vert f(t_0)\Vert <\Vert f(s)\Vert $ for every $s\in \mathcal{O}$, the function $h$ defined on $K$ by $h(s):=f(s)$ if $s\in K\setminus \mathcal{O}$ and $h(s):=\frac{f(s)}{\Vert f(s)\Vert}$ if $s\in \mathcal{O}$  belongs to $S_{C_0(K, X)}$,  and $\Vert h(t_0)\Vert =1$. Therefore, $$\Vert f-h\Vert =\sup_{s\in \mathcal{O}}\big\Vert \frac{f(s)}{\Vert f(s)\Vert}-f(s)\big\Vert = \sup_{s\in \mathcal{O}}\Vert 1-\Vert f(s)\Vert\Vert \leq 1-\frac{\Vert f(t_0)\Vert}{2}.$$
	Now, $h\in  A(t_0,\frac{f(t_0)}{\Vert f(t_0)\Vert}) $,  and for  $\psi $ in ${\rm supp}(t _0,\frac{f(t_0)}{\Vert f(t_0)\Vert})$ it follows that $$ 1=\psi (\Delta (h)) -\psi (\Delta (f))\leq \Vert \Delta (h) -\Delta (f)\Vert =\Vert h -f\Vert \leq 1-\Vert f(t_0)\Vert<1 ,$$ the desired contradiction, hence $f(t_0)= 0$.\end{proof}

\medskip

\section{The Mazur--Ulam property in $\bigoplus_{\gamma \in \Gamma}^{c_0} X_\gamma $ and  $\bigoplus_{\gamma \in \Gamma}^{\ell_{\infty}} X_\gamma $}

\medskip

We recall that an element $x\in S_X$ is a \emph{smooth point} of $X$  if there is a unique $f\in S_{X^*}$ such that $f(x)=1$. We denote by $\it{Sm}(X)\subseteq S_X$ the set of all smooth points of $X$ and, given $x\in \it{Sm}(X)$,   we denote by $\varphi_x\in S_{X^*}$   the unique functional such that $\varphi_x (x)=1$. It is known that $x$ is a smooth point of $X$ if and only if the norm of $X$  is G\^{a}teaux differentiable at $x$ \cite[Corollary 1.5]{DevGodZiz93}. We refer to \cite[\S I.1]{DevGodZiz93} for the basic results on G\^{a}teaux differentiability.

We note that the class of those Banach spaces $X$ such that $ \it{Sm}(X)$ is norm dense in $S_X$ contains all separable Banach spaces \cite[Proposition 9.4.3]{Ro} and all Asplund Banach spaces \cite[\S I.1]{DevGodZiz93}.

\smallskip

\begin{proposition}\label{c-0-suma-estric-convex-norming} Let  $\{X_\gamma \}_{\gamma \in \Gamma}$ be a family of strictly convex Banach spaces such that  $\vert \Gamma \vert \geq 2$, and such that dim$(X_\gamma )\geq 2 $ and the set $\it{Sm}(X_\gamma)$  is norm dense in $S_{X_\gamma }$ for every $\gamma $. Let $Y$ be a Banach space, and let $\Delta : S_{Z}\to S_Y$ be a surjective isometry such that $\Delta \arrowvert_{ A(\gamma ,x)}$ is an affine map whenever $\gamma \in \Gamma $ and  $x\in S_{X_{\gamma }}$. Then the equality $$\psi \Delta (z) = \varphi_{x} \otimes P_{\gamma} (z) $$ holds for all $\gamma \in \Gamma $, $x\in S_{X_{\gamma }}$,  $\psi\in {\rm supp}(\gamma ,x)$ and $z \in S_Z$ with $\Vert z(\gamma ' )\Vert =1$ for some $\gamma ' \neq \gamma $.
\end{proposition}

\begin{proof}   Let us fix $\gamma_0, \gamma_1\in \Gamma$ with $\gamma_1 \neq \gamma _0$, $x_0\in \it{Sm}(X_{\gamma_0})$, $\psi\in {\rm supp}(\gamma_0,x_0)$, and $p\in S_Z$ with $\Vert p(\gamma_1 )\Vert =1$.  We set $Z_{\gamma_1} :=P_{\Gamma \setminus {\gamma_1} }(Z)$ and define $p_1\in Z$ by $p_1 (\gamma_1):=p(\gamma_1 )$ and $p_1(\gamma ):=0$ if $\gamma \neq \gamma_1$. It follows that $$ A(\gamma_1 ,p(\gamma_1 )) = \{ z\in S_{Z} : z(\gamma_1 ) = p(\gamma_1)\} = p_1 + B_{Z_{\gamma_1 }}.$$ Since $\Delta : A(\gamma _1,p(\gamma _1))\to \Delta (A(\gamma_1 ,p(\gamma _1)))$ is affine,  the mapping \linebreak  $\Delta_{p_1} : B_{Z_{\gamma_1}}\to \Delta (A(\gamma_1 ,p(\gamma_1 )))$ defined by $\Delta_{p_1}(z):=\Delta (p_{1}+z)$  is affine. Now, the mapping $\varphi$ from $B_{Z_{\gamma_1}}$ to $\mathbb{R}$ defined by  $$\varphi (z):= \psi \Delta (p_1+z) - \psi \Delta(p_1)$$ is  affine with $|\varphi (z) |\leq 1$ for all $z\in  B_{Z_{\gamma_1}}$. We can regard $\varphi$ as the restriction to $B_{Z_{\gamma_1}}$ of a functional $\varphi '$ in $ B_{Z_{\gamma_1}^*}$. Let $\widetilde{\varphi}$ be the function from $Z$ to $\RR $ defined by $\widetilde{\varphi} (z) := \varphi '(P_{\Gamma \setminus \gamma_1 }(z))$ for every $z\in Z$. It follows that $\widetilde{\varphi}$ belongs to $B_{Z^*}$. By Lemma \ref{l tech 2} (6), $\psi\Delta(p_1)=0$ since $p_1 (\gamma_0) =0$ and $\psi\in {\rm supp}(\gamma_0,x_0)$, and hence $$\widetilde{\varphi} (z) = \psi \Delta (p_{\gamma _1}+ P_{\Gamma \setminus \gamma_1 }(z))\  \mbox{ for every} \ z\in B_{Z}.$$
We define $z_0\in S_{Z_{\gamma_1}}\subseteq S_{Z}$  by $z_0(\gamma _0):=x_0$ and $z_0(\gamma ):=0 $  if $\gamma \neq \gamma_0$. It is clear that $p_1 +z_0\in A(\gamma _0,x_0)$, and since $\psi\in {\rm supp}(\gamma_0,x_0)$, we get $\widetilde{\varphi} (z_0) = \psi\Delta (p_1+z_0 )=1,$ and so $\|\widetilde{\varphi}\|=1$ in $Z^*$.

We claim that \begin{equation}\label{igualdad} \widetilde{\varphi} (z) = \varphi_{x_0}(z(\gamma_0)) \ \ \mbox{for every}\ z\in Z.
	\end{equation}
Let us fix $\varepsilon>0$ and $z\in B_{Z}$. Since  $X_{\gamma_0}$ is smooth at $x_0$, it follows from \cite[Theorem I.1.4]{DevGodZiz93} that there exists $0<\rho<\frac{\varepsilon}{8}$ with the following property: \begin{equation}\label{Gateaux} |\phi (z(\gamma_0)) -\varphi_{x_0} (z(\gamma_0)) |< \frac{\varepsilon}{4}\ \mbox{ for every}\  \phi\in B_{X^*_{\gamma_0}}\  \mbox{with}\   \phi (x_0)>1-\rho .
	\end{equation}
Since $\it{Sm}(X_\gamma )$ is norm dense in $S_{X_\gamma}$,  for each $\gamma \in \Gamma$, the set $$\{ \varphi_x \otimes P_\gamma : \gamma \in \Gamma,\ \ x\in \it{Sm}(X_\gamma ) \}\subseteq B_{Z^*}$$ is a norming set, and therefore, by the Hahn--Banach theorem, \begin{equation}\label{convex-hull} \overline{co}^{w^*} \{ \varphi_x \otimes P_\gamma : \gamma \in \Gamma,\ \ x\in \it{Sm}(X_\gamma ) \} =B_{Z^*}. \end{equation}
	By (\ref{convex-hull}), there exist $\lambda_1,\ldots ,\lambda _n\in (0,1]$ with $\sum_{i=1}^n\lambda_i=1$, $\gamma_i\in \Gamma$, and $x_{\gamma_i}\in \it{Sm}(X_{\gamma _i})$   such that
	\begin{equation}\label{smoot-2} \left|\widetilde{\varphi} (u) - \!  \! \left(\!\sum_{i=1}^{n} \! \lambda_i \varphi_{x_{\gamma _i}}\! \! \otimes \! \! P_{\gamma_i}\!  \!\right)\! (u) \right| <{\rho},\ \mbox{for}  \ u\in \{z_0,z \} .\end{equation}
	
	Set $N:=\{ i\in \{1,\ldots ,n \}: \gamma_i=\gamma_0 \}$. Since $\widetilde{\varphi} (z_0) =1$,  we have that $$ \left|1 -\left(\sum_{i=1}^{n} \lambda_i \ \varphi_{x_{\gamma _i}}\otimes P_{\gamma_i}\right) (z_0) \right|
	=  \left|\left(\varphi_{x_0} -\sum_{i\in N} \lambda_i \  \varphi_{x_{i}}\right)  (x_0) \right|<\rho \ .$$ It follows from (\ref{Gateaux}) that \begin{equation}\label{eq 6b} \left|\left(\varphi_{x_0} - \sum_{i\in N}\lambda_i \ \varphi_{x_i}\right) z(\gamma_0)  \right|< \frac{\varepsilon}{4}.
	\end{equation} On the other hand, we have $ 1- \sum_{i\in N}\lambda_i \  < {\rho},$ and so  $\sum_{i\notin N}\lambda_i <\rho$ because $ \sum_{i=1}^n \lambda_i =1$ . Now,  for each $i\in N$, we set $\displaystyle \mu_i := \frac{\lambda_i }{\sum_{j\in N}    \lambda_j }$.

	 It is not hard to check that $$\left\| \sum_{i\in N}\lambda_i  \ \varphi_{x_i}\otimes P_{\gamma_0} -  \sum_{i\in N}\mu_i \varphi_{x_{i}}\otimes P_{\gamma_0} \right\| < {\rho},$$ and hence $$\left\| \sum_{i=1}^{n} \lambda_i \ \varphi_{x_{\gamma_i}}\otimes P_{\gamma_i}- \sum_{i\in N}  \mu_i \varphi_{x_{i}}\otimes P_{\gamma_0} \right\|<2 {\rho}.$$ Keeping in mind the above inequalities we obtain $$\left|\left( \varphi_{x_{0}}  - \sum_{i\in N} \mu_i\ \varphi_{x_{i}} \right)  (z(\gamma_0))\right| \leq  \left|\left(\varphi_{x_0} - \sum_{i\in N} \lambda_i \ \varphi_{x_{i}}\right) (z(\gamma_0))  \right| $$
	$$ + \left|\left(\sum_{i\in N} \lambda_i \ \varphi_{x_{i}} - \sum_{i\in N} \mu_i\ \varphi_{x_{i}}\right) (z(\gamma_0))  \right| <\hbox{(by (\ref{eq 6b})})$$ $$<\frac{\varepsilon}{4}+ \left| \left( \sum_{i\in N} \lambda_i  \ \varphi_{x_{i}}\otimes P_{\gamma_0} -  \sum_{i\in N} \mu_i \varphi_{x_{i}}\otimes P_{\gamma_0}\right) (z(\gamma_0)) \right|<$$ $$< \frac{\varepsilon}{4}+ {\rho}.$$
	
	We therefore have
	
	$$ \left| \widetilde{\varphi} (z) - \varphi_{x_0}\otimes P_{\gamma_0} (z) \right| \leq  \left| \widetilde{\varphi} (z) -\left(\sum_{i=1}^n \lambda_i  \varphi_{x_{\gamma _i}}\otimes P_{\gamma_i}\right) (z) \right|$$

	$$ + \left| \left(\sum_{i=1}^n \lambda_i  \varphi_{x_{\gamma_i}}\otimes P_{\gamma_i}- \sum_{i\in N} \mu_i \varphi_{x_{i}}\otimes P_{\gamma_0}\right) (z) \right| $$

	$$ + \left| \left( \sum_{i\in N} \mu_i \varphi_{x_{i}}\otimes P_{\gamma_0} - \varphi_{x_0}\otimes P_{\gamma_0} \right) (z) \right|< \hbox{(by (\ref{smoot-2}))}$$
	$$< {\rho} + \left\| \sum_{i=1}^{n} \lambda_i  \varphi_{x_{\gamma_i}} \! \otimes P_{\gamma_i}-\! \!  \sum_{i\in N} \mu_i \varphi_{x_{i}}\! \otimes P_{\gamma_0} \right\| + \left|\left( \! \varphi_{x_0}  - \! \sum_{i\in N}\!  \mu_i \varphi_{x_i} \right)  (z(\gamma_0))\right|$$ $$< 4 {\rho} +\frac{\varepsilon}{4}<\varepsilon.$$ 	By letting $\varepsilon \to 0$, we realize that $ \widetilde{\varphi} (z) = \varphi_{x_0}\otimes P_{\gamma_0} (z)$, which shows that (\ref{igualdad}) holds. Consequently, $$\psi \Delta (p)= \psi  \Delta (p_1+P_{\Gamma \setminus \gamma_1 }(p))=\widetilde{\varphi}(p)= \varphi_{x_0} \otimes P_{\gamma_0} (p)$$ as desired. \end{proof}

\begin{proposition}\label{afin} Let  $\{X_\gamma \}_{\gamma \in \Gamma}$ be a family of strictly convex Banach spaces such that  $\vert \Gamma \vert \geq 2$, and such that dim$(X_\gamma ) \geq 2 $ and the set $\it{Sm}(X_\gamma)$  is norm dense in $S_{X_\gamma }$ for every $\gamma $. Let $Y$ be a Banach space, and let $\Delta : S_{Z}\to S_Y$ be a surjective isometry. Then, for each $\gamma \in \Gamma $ and each $x\in \it{Sm}(X_\gamma )$, the equality $$\psi \Delta (z) = \varphi_{x} \otimes P_{\gamma} (z) $$ holds for all $\psi\in {\rm supp}(\gamma ,x)$ and $z\in S_Z$.
\end{proposition}
\begin{proof} We first show the result in the case $Z=Z_{\infty}$.
	
Let us fix $\gamma \in \Gamma $ and $x\in \it{Sm}(X_\gamma )$.  Since the set $$B:=\{z\in S_{Z_\infty}: \ \mbox{there exists} \ \gamma \in \Gamma \ \mbox{such that} \ \Vert z(\gamma )\Vert _{X_\gamma}=1 \}$$ is dense in $S_{Z_\infty}$, to prove the  equality  $$\psi \Delta (z) = \varphi_{x} \otimes P_{\gamma} (z) \ \ \mbox{for all}\ \psi\in {\rm supp}(\gamma ,x)\ \mbox{and } z\in S_{Z_{\infty}},$$ it is enough to verify it in the case that $z$ lies in $B$, and hence there exists $\gamma _0\in \Gamma$ such that $\Vert z(\gamma _0 )\Vert _{X_{\gamma _0}}=1$. We define $z_0\in S_{Z_\infty}$,  by $z_0(\gamma_0)=z(\gamma_0)$ and $z_0(\gamma )=0$ for $\gamma \neq \gamma_0$, and $Z_{\gamma_0}:= \bigoplus _{\gamma \in \Gamma \setminus \gamma_0}^{\ell_{\infty}} X_\gamma $.
By Lemma \ref{maximal}, $A(\gamma_0 ,z(\gamma_0 )) $ is a maximal convex subset of $S_{Z_\infty}$. Then, by \cite[Lemma 5.1]{ChenDong2011},   $\Delta (A(\gamma_0 ,z(\gamma_0 )))$ is a maximal convex subset of $Y$. Furthermore, the mapping $\Delta_{z_0} : B_{Z_{\gamma_0}}\to \Delta (A(\gamma_0 ,p(\gamma_0 )))$ defined by $\Delta_{z_0}(z):=\Delta (z_0+z)$ is a surjective isometry.  It follows from  Proposition \ref{Strong Mankiewicz property} and Theorem \ref{Mori-Ozawa} that  $\Delta_{z_0}$ is an affine map. Then the restriction of	$\Delta$ to $A(\gamma_0 ,z(\gamma_0 ))$, regarded as a mapping onto $\Delta (A(\gamma_0 ,z(\gamma_0 )))$, is affine. By Lemma \ref{extreme points}, $\|p(\gamma )\|_{X_{\gamma }}=1$ for all $p\in Ext (B_{Z_\infty})$ and $\gamma \in \Gamma$, and  hence, by Proposition \ref{c-0-suma-estric-convex-norming},  the equality $$\psi \Delta (p) = \varphi_{x} \otimes P_{\gamma } (p) $$ is true for every $\psi\in {\rm supp}(\gamma ,x)$. By Proposition \ref{convex-hull-extreme-point}, there exists $p,q\in Ext(B_{Z_\infty}) $ such that $z=\frac{1}{2}(p+q)$. Since $X_{\gamma _0}$ is strictly convex and $z(\gamma_0)=\frac{1}{2}(p(\gamma_0)+q(\gamma_0))$, it follows that $p,q\in A(\gamma_0 ,z(\gamma_0 ))$. Since $\Delta (z)=\frac{1}{2}(\Delta (p)+\Delta (q))$, we have that  $$\psi \Delta (z)=\frac{1}{2}(\psi \Delta (p)+\psi \Delta (q))=$$ $$= \frac{1}{2} (\varphi_{x} \otimes P_{\gamma } (p)+ \varphi_{x} \otimes P_{\gamma } (q))=\varphi_{x} \otimes P_{\gamma } (\frac{1}{2}(p+q))=\varphi_{x} \otimes P_{\gamma } (z)$$ for every $\psi\in {\rm supp}(\gamma ,x).$

Now we  show the result in the case $Z=Z_0$. We  begin by showing that, given $\gamma_0 \in \Gamma $ and $z\in S_{Z_0}$ with $\Vert z(\gamma_0  )\Vert _{X_{\gamma_0 }}\! \!=1$, the restriction of	$\Delta$ to $A(\gamma_0 ,z(\gamma_0 ))$, regarded as a mapping onto $\Delta (A(\gamma_0 ,z(\gamma_0 )))$, is affine.

By Lemma \ref{maximal}, $A(\gamma_0 ,z(\gamma_0 )) $ is a maximal convex subset of $S_{Z_0}$. Then, by \cite[Lemma 5.1]{ChenDong2011}, $\Delta (A(\gamma_0 ,z(\gamma_0 )))$ is a maximal convex subset of $Y$. Let us take $\varepsilon >0$, $z_1,z_2\in A(\gamma_0 ,z(\gamma_0 ))$ and $\alpha \in (0,1)$. Then there exists a finite subset $\Gamma _0\subseteq \Gamma$ with $\vert \Gamma_0\vert \geq 2$, such that $\gamma _0\in \Gamma _0$ and $\Vert z_i-z_i'\Vert <\varepsilon$  where $z_i':=P_{\Gamma_0}(z_i) $ ($i=1,2$).
We set  $$A_{\Gamma_0}(\gamma_0 ,z(\gamma_0 )):=\{z\in A(\gamma_0 ,z(\gamma_0 )): P_{ \Gamma \setminus \Gamma_0}(z)=0 \},$$ and define  $z_0\in S_{Z_0}$  by $z_0(\gamma_0):=z(\gamma_0)$ and $z_0(\gamma ):=0$ for every $\gamma \in \Gamma \setminus \gamma_0$. It is clear that $$A_{\Gamma_0}(\gamma_0 ,z(\gamma_0 ))=z_0+B_{\bigoplus _{\gamma \in \Gamma_0 \setminus \gamma_0}^{\ell_\infty  } X_\gamma },$$ and that $z_1',z_2'\in A_{\Gamma_0}(\gamma_0 ,z(\gamma_0 ))$.

Let us  verify that the set $\Delta (A_{\Gamma_0}(\gamma_0 ,z(\gamma_0 )))$ is a convex subset of $S_Y$. Let $a,b\in \Delta (A_{\Gamma_0}(\gamma_0 ,z(\gamma_0 )))\subseteq \Delta (A(\gamma_0 ,z(\gamma_0 )))$ and $\beta \in (0,1)$. Then $\beta a+(1-\beta )b\in \Delta (A(\gamma_0 ,z(\gamma_0 )))$. Take $u,v\in A_{\Gamma_0}(\gamma_0 ,z(\gamma_0 ))$ and $w\in A(\gamma_0 ,z(\gamma_0 ))$ such that $\Delta (u)=a$, $\Delta (v)=b$, and $\Delta (w)=\beta a+(1-\beta )b$. By Lemma \ref{l tech 2} (6), for all $\gamma \in \Gamma \setminus \Gamma_0$ and $x\in S_{X_\gamma}$ we have   $\psi \Delta(u) =0 =\psi \Delta(v) $ for every $\psi \in {\rm supp}(\gamma ,x)$. This implies that $\psi \Delta(w)=0$ for all $\gamma \in \Gamma \setminus \Gamma_0$, $x\in S_{X_\gamma}$, and $\psi \in {\rm supp}(\gamma ,x)$. By Lemma \ref{l tech 2} (7), $w(\gamma)=0$ for every $\gamma \in \Gamma \setminus \Gamma_0$. It follows that $w\in  A_{\Gamma_0}(\gamma_0 ,z(\gamma_0 ))$, and hence $\Delta (A_{\Gamma_0}(\gamma_0 ,z(\gamma_0 )))$ is a convex subset of $S_Y$, as desired. Furthermore, the mapping $\Delta_{z_0} : B_{\bigoplus _{\gamma \in \Gamma_0 \setminus \gamma_0}^{\ell_\infty } X_\gamma }\to \Delta (A_{\Gamma_0}(\gamma_0 ,z(\gamma_0 )))$ defined by $\Delta_{z_0}(z):=\Delta (z_0+z)$ is a surjective isometry.  Therefore, by  Proposition \ref{Strong Mankiewicz property} and Theorem \ref{Mori-Ozawa},  $\Delta_{z_0}$ is an affine map, and so the restriction of	$\Delta$ to $A_{\Gamma_0}(\gamma_0 ,z(\gamma_0 ))$, regarded as a mapping onto $\Delta (A_{\Gamma_0}(\gamma_0 ,z(\gamma_0 )))$, is affine.
 This fact, together with the inequalities  $$\Vert \Delta (\alpha z_1+(1-\alpha )z_2)- \Delta (\alpha z'_1+(1-\alpha )z'_2)\Vert <\varepsilon $$ and $$\Vert \alpha \Delta (z_1)+(1-\alpha)\Delta (z_2) -\alpha \Delta (z'_1)-(1-\alpha)\Delta (z'_2)\Vert <\varepsilon ,$$ allows us to obtain $$\Vert \Delta (\alpha z_1+(1-\alpha )z_2)-\alpha \Delta (z_1)-(1-\alpha)\Delta (z_2)  \Vert <2\varepsilon .$$ 	By letting $\varepsilon \to 0$, we realize that the restriction of	$\Delta$ to $A(\gamma_0 ,z(\gamma_0 ))$, regarded as a mapping onto $\Delta (A(\gamma_0 ,z(\gamma_0 )))$, is affine.

Now the proof shall be concluded by showing that, for  $x_0\in \it{Sm}(X_{\gamma _0})$,  the equality $$\psi \Delta (z) = \varphi_{x_0} \otimes P_{\gamma_0} (z) $$ is true for all $\psi\in {\rm supp}(\gamma _0,x_0)$ and $z \in S_{Z_0}$.

Let us fix $z \in S_{Z_0}$ and $\gamma_1\in \Gamma$ such that $\Vert z(\gamma_1 )\Vert _{X_{\gamma_1 }}=1$. Pick $\varepsilon >0$, $\gamma _2\in \Gamma \setminus \{\gamma_0 \}$ with $\Vert z(\gamma_2 )\Vert _{X_{\gamma_2 }}<\varepsilon$, and  $x_{\gamma_2}\in S_{X_{\gamma_2 }}$. For $i=1,2$, define $p_i\in S_{Z_0}$ by $p_i(\gamma):=z(\gamma)$ if $\gamma \neq \gamma_2$ and $p_i(\gamma_2):=(-1)^ix_{\gamma_2}$. Then $p_i\in A(\gamma_1 ,z(\gamma_1))$ ($i=1,2$) and $\Vert z-\frac{p_1+p_2}{2} \Vert <\varepsilon$. By Proposition \ref{c-0-suma-estric-convex-norming}, $$\psi \Delta (p_i) = \varphi_{x_0} \otimes P_{\gamma_0 } (p_i) $$  for every $\psi\in {\rm supp}(\gamma_0 ,x_0)$. Since $\Delta \arrowvert_{ A(\gamma_1 ,z(\gamma_1 ))}$ is an affine map,  we have that $$\psi \Delta (\frac{p_1+p_2}{2}) =\frac{\psi \Delta (p_1)+\psi \Delta (p_2)}{2}= \varphi_{x_0} \otimes P_{\gamma_0 } (\frac{p_1+p_2}{2}) .$$ Then $$\Vert \psi \Delta (z) - \varphi_{x_0} \otimes P_{\gamma_0} (z) \Vert $$ $$ =\Vert \psi \Delta (z)-\psi \Delta (\frac{p_1+p_2}{2})\Vert +\Vert \varphi_{x_0} \otimes P_{\gamma_0 } (\frac{p_1+p_2}{2}) - \varphi_{x_0} \otimes P_{\gamma_0} (z) \Vert <2\varepsilon . $$ 	By letting $\varepsilon \to 0$, we obtain \[\psi \Delta (z) = \varphi_{x_0} \otimes P_{\gamma_0} (z). \qedhere \]	
\end{proof}

We can now establish the main result of this section.

\begin{theorem}\label{suma-Mazur-Ulam} Let  $\{X_\gamma \}_{\gamma \in \Gamma}$ be a family of strictly convex Banach spaces such that  $\vert \Gamma \vert \geq 2$, and such that dim$(X_\gamma ) \geq 2 $ and  the set $\it{Sm}(X_\gamma)$  is norm dense in $S_{X_\gamma }$ for all $\gamma $. Then $Z_0$ and $Z_\infty$ satisfy the  Mazur--Ulam property.
\end{theorem}
\begin{proof} Let $Y$ be a Banach space, and let $\Delta : S_{Z}\to S_Y$ be a surjective isometry. By Proposition \ref{afin}, the equality $$\psi \Delta (z) = \varphi_{x} \otimes P_{\gamma } (z) $$ is true for all $\gamma \in \Gamma $, $x\in \it{Sm}(X_\gamma )$, $\psi\in {\rm supp}(\gamma ,x)$, and $z\in S_Z$. Now, the set $$\{ \varphi_x \otimes P_\gamma : \gamma \in \Gamma,\ \ x\in \it{Sm}(X_\gamma ) \}\subseteq B_{Z^*}$$ is a norming set for $Z$, and hence the set $$\bigcup \{ {\rm supp}(\gamma ,x) : \gamma \in \Gamma  ,  x\in \it{Sm}(X_\gamma ) \}\subseteq B_{Y^*}$$ is a norming set for $Y$. It follows that $$\Vert \Delta (z) +\lambda \Delta (\tilde{z})\Vert $$ $$= \sup  \{ \psi( \Delta (z)) +\lambda \psi(\Delta (\tilde{z}))  : \gamma \in \Gamma,\ \ x\in \it{Sm}(X_\gamma ) , \  \psi\in {\rm supp}(\gamma ,x)\}$$ $$=\sup  \{ \varphi_{x} \otimes P_{\gamma }(z) +\lambda \varphi_{x} \otimes P_{\gamma }(\tilde{z}))  :\gamma \in \Gamma,\ \ x\in \it{Sm}(X_\gamma ) \}=\Vert z+\lambda \tilde{z} \Vert  $$ for all $\lambda >0$ and $z,\tilde{z}\in S_{Z}$. 	By \cite[Lemma 2.1]{FangWang06}, $\Delta$  can be extended to a real-linear isometry of $Z$ onto $Y$.\end{proof}

Now, let us see that the hypothesis about the G\^{a}teaux differentiability of the norm in   Theorem \ref{suma-Mazur-Ulam} can be removed in the   particular cases of  countable  $c_0$-sums and finite $\ell_\infty$-sums.

\begin{lemma}\label{separable} Let $X_1$, $X_2$, and $Y$ be Banach spaces, and let $A$ be a separable subset of $X_1\oplus_\infty X_2$. Suppose that  $\Delta :S_{X_1\oplus_\infty X_2} \to S_Y$ is a surjective isometry. Then there exist separable subspaces $M_1\subseteq X_1$, $M_2\subseteq X_2$ and $N\subseteq Y$ such that $A\subseteq M_1\oplus_\infty M_2$ and $\Delta (S_{M_1\oplus_\infty M_2})= S_N$.
\end{lemma}
\begin{proof} Since $A_1:=A$ is a separable subset of $X_1\oplus_\infty X_2$, there exists a countable subset of $A_1$ (say $\{ a_n:n\in \NN \}$) dense in $A_1$. For each $n\in \NN$, we have $a_n=a_n^1+a_n^2$ with $a_n^1\in X_1$ and $a_n^2\in X_2$. Now,  for $i=1,2$, consider the separable closed subspace $B_{(i,1)}\subseteq X_i$ defined  by $B_{(i,1)}:=\overline{Lin}\{ a_n^i \} _{n\in \NN }$. It is clear that  $A_1\subseteq B_{(1,1)}\oplus_\infty B_{(2,1)}\subseteq X_1\oplus_\infty X_2.$ Now, $B_{(1,1)}\oplus_\infty B_{(2,1)}$ is a separable closed subspace of $X_1\oplus_\infty X_2$, hence $\Delta (S_{B_{(1,1)}\oplus_\infty B_{(2,1)}})$ is a separable subset of $S_Y$. Define $N_1:=\overline{Lin}(\Delta (S_{B_{(1,1)}\oplus_\infty B_{(2,1)}}))$. Since $S_{N_1}$ is a separable subset  of $S_Y$, we have that $A_2:=\Delta ^{-1}(S_{N_1})$ is a separable subset of $S_{X_1\oplus_\infty X_2}$.
Consider  $B_{(i,2)}\subseteq X_i$ such that  $B_{(i,1)}\subseteq B_{(i,2)}$ ($i=1,2$) and $A_2\subseteq B_{(1,2)}\oplus_\infty B_{(2,2)}$ and $N_2=\overline{Lin}(\Delta (S_{B_{(1,2)}\oplus_\infty B_{(2,2)}}))$. Define the separable closed subspaces $B_{(i,k)}\subseteq X_i$ ($i=1,2$) and $N_k\subseteq Y$ inductively by $\Delta ^{-1}(S_{N_k})\subseteq S_{B_{(1,k+1)}\oplus_\infty B_{(2,k+1)}}$, $B_{(i,k)}\subseteq B_{(i,k+1)}$ ($i=1,2$)  and $N_k:=\overline{Lin}(\Delta (S_{B_{(1,k)}\oplus_\infty B_{(2,k)}}))$ for every $k\geq 2$. We consider the separable closed subspaces $M_i:=\overline{\bigcup_{k\in \NN}B_{(i,k)}}$   ($i=1,2$) and $N:=\overline{\bigcup_{k\in \NN}N_k}$. The inclusion $\Delta (S_{M_1\oplus_\infty M_2})\subseteq S_N$ is clear. Let $y$ be in $\bigcup_{k\in \NN }N_k$ with $\Vert y\Vert =1$. Then there exists $k\in \NN$ such that $y\in S_{ N_k}$, and hence $\Delta ^{-1}(y)\in S_{B_{(1,k+1)}\oplus_\infty B_{(2,k+1)}} \subseteq S_{M_1\oplus_\infty M_2}$. Now let $y$ be arbitrary in $S_N$. Then we can find  a sequence $\{ y_n\}$ of norm-one elements in  $\bigcup_{k\in \NN}N_k$ such that $\{y_n\}\rightarrow y$. It is clear that  $\{\Delta ^{-1}(y_n)\}$ is a Cauchy sequence in $S_{M_1\oplus_\infty M_2}$, hence there exists $z\in S_{M_1\oplus_\infty M_2}$ such that $\{\Delta ^{-1}(y_n)\}\rightarrow z$. We conclude that $\Delta (z)=y$.	Therefore $\Delta (S_{M_1\oplus_\infty M_2})\supseteq S_N$.
\end{proof}

\begin{lemma}\label{separablec_0} Let $\{X_n \}_{n\in \NN }$ be a family of  Banach spaces and let $A$ be a separable subset of $Z_0$. Let $Y$ be a Banach space, and let $\Delta : S_{Z_0}\to S_Y$ be a surjective isometry. Then there exist separable subspaces $M_n \subseteq X_n$  ($n\in \NN$),  and $N\subseteq Y$ such that $A\subseteq \bigoplus_{n\in \NN  }^{c_0} M_n $  and $\Delta (S_{\bigoplus_{n\in \NN }^{c_0} M_n })= S_N$.
\end{lemma}
\begin{proof} Since $A$ is a separable subset of $\bigoplus_{n\in \NN }^{c_0} X_n $, there exists a countable subset of $A$ (say $\{ a_k:  k\in \NN \}$) dense in $A$. For each $n\in \NN$, $P_n(\{ a_k:   k\in \NN \})$ is a countable subset of $X_n$, and hence $$B_n:=\overline{Lin}P_n(\{ a_k: k\in \NN \})$$ is a separable subspace of $X_n$. It is  clear that $A\subseteq \bigoplus_{n\in \NN }^{c_0} B_n $. For $n\in \NN$, let us fix  a dense countable subset $D_n$ of $B_n$ containing $0$. Then the subset $\mathcal F$ of $\bigoplus_{n\in \NN }^{c_0} B_n $ consisting of all finitely  supported vectors whose $n$-th  coordinate is in $D_n$ for each $n$, is  countable and dense in $\bigoplus_{n\in \NN }^{c_0} B_n $. Therefore  $\bigoplus_{n\in \NN }^{c_0} B_n $ is a separable subspace of $\bigoplus_{n\in \NN }^{c_0} X_n $. We conclude the proof arguing as in the proof of the previous lemma.
\end{proof}

Now, we are ready to prove the following.

\begin{theorem} \label{main} Let $\{X_n \}_{ n \in \NN }$  be a  family of strictly convex Banach spaces such that dim$(X_n )\geq 2 $ for all $n $. Then the Banach spaces $Z_n:=X_{1} \oplus _\infty \cdots \oplus _\infty X_{n} $ $($$n \geq 2$$)$ and $\bigoplus_{n \in \NN }^{c_0} X_n $ satisfy the Mazur--Ulam property.
\end{theorem}
\begin{proof}  Let  $Y$ be a Banach space, and let $\Delta : S_{Z_n}\to S_Y$ be a surjective isometry. Let us take $z,\tilde{z}\in S_{Z_n}$ and $\lambda >0$. By Lemma \ref{separable}, there exist separable subspaces $M_i\subseteq X_i$ for all $i=1, \cdots ,n$ and $N\subseteq Y$ such that $z, \tilde{z}\in  M_1\oplus_\infty \cdots \oplus_\infty M_n$ and $$\Delta ( S_{M_1\oplus_\infty \cdots \oplus_\infty M_n)} =S_N.$$ Since $X_i$ is a strictly convex  with dim$(X_i)\geq 2$, we  have that  $M_i$ is  a strictly convex  with dim$(M_i)\geq 2$ for all $i=1, \ldots ,n$. According to Mazur's theorem \cite[Proposition 9.4.3]{Ro}, $\it{Sm}(M_1\oplus_\infty \cdots \oplus_\infty M_n)$  is norm dense in $S_{M_1\oplus_\infty \cdots \oplus_\infty M_n }$. By Theorem \ref{suma-Mazur-Ulam}, the restriction of	$\Delta$ to $S_{M_1\oplus_\infty \cdots \oplus_\infty M_n}$, regarded as a mapping onto $S_N$,  admits a unique extension to a surjective real-linear isometry $\Phi$ from $M_1\oplus_\infty \cdots \oplus_\infty M_n$ to $N$, and so $$\Vert \Delta (z) +\lambda \Delta (\tilde{z})\Vert = \Vert \Phi (z +\lambda \tilde{z})\Vert =\Vert z+\lambda \tilde{z} \Vert .$$
Since  $z,\tilde{z}\in S_{Z_n}$ and $\lambda >0$ are arbitrary, it follows from \cite[Lemma 2.1]{FangWang06} that $\Delta$ can be extended to a real-linear isometry from $Z_n$ to $Y$. By replacing Lemma \ref{separable} with Lemma \ref{separablec_0}, the proof in the case $\bigoplus_{\gamma \in \Gamma}^{c_0} X_\gamma $ is similar.	
\end{proof}

Using the fact that $c_0(\Gamma )$ admits a strictly convex equivalent norm, one can realize that  a Banach space $X$  is strictly convex renormable as soon as there exists a  one-to-one bounded linear operator $T:X\to c_0(\Gamma )$. In this case, the equivalent  strictly convex norm in $X$ is given by $$\vert \vert \vert x\vert \vert \vert 	 = \Vert x\Vert + \Vert T(x)\Vert \ \ (x\in X).$$  This fact implies that  every weakly compactly generated space (in particular every separable space) admits a strictly convex equivalent renorming (see \cite{DevGodZiz93}, \cite{MoOrTrZi} and \cite{ArMe}).  Let us note that there are strictly convex renormable Banach spaces which cannot be linearly and continuously imbedded into $c_0(\Gamma )$ for any $\Gamma$ \cite{DaLi}.

\begin{corollary} \label{renorm} Let $X$ be a strictly convex  Banach space. Then $X$ has a equivalent norm with the Mazur--Ulam property.
\end{corollary}

\begin{proof} The result is clear for finite-dimensional Banach spaces since every Hilbert space   satisfies the Mazur--Ulam property \cite[Proposition 4.5]{BeCuFePe}. So that we can assume that $X$ is an infinite-dimensional  Banach space. Let $M$ be a finite dimensional subspace of $X$ with dim$(M)\geq  2$. Since $M$ is a complemented subspace, there exists a closed subspace $N$ of $X$ such that  $X=M\oplus N$. Since $X$ is strictly convex, then $M,N$ are strictly convex. We conclude the proof  applying	Theorem \ref{main} to $X$ with the equivalent norm given by $X=M\oplus_\infty N$.
\end{proof}

\begin{corollary}\label{renorm1} Let $X$ be a  Banach space such that $X^*$ is strictly convex. Then $X$ has a equivalent norm whose dual norm has the Mazur--Ulam property.
\end{corollary}
\begin{proof} We can assume that $X$ is infinite-dimensional. Let $M$ be a finite dimensional subspace of $X$ with dim$(M)\geq  2$. We consider the equivalen norm in $X$ given by $X=M\oplus_1 N$. Then $X^*=M^*\oplus_\infty N^*$. Since $N^*$ and $M^*$ are strictly convex, it follows from Theorem \ref{main} that $X^*$ has the Mazur--Ulam property.
\end{proof}

We recall that, if $X$ is a  weakly compactly generated Banach space, then there exists $\Vert .\Vert _1$ and $\Vert .\Vert _2$ equivalent norms in $X$ such that $(X,\Vert .\Vert _1)$ and $(X,\Vert .\Vert _2)^*$  are strictely convex \cite[Theorems VI.2.1 and VII.1.6]{DevGodZiz93}.

\begin{corollary}\label{renorm2} Let $X$ be a  weakly compactly generated Banach space. Then there exists $\Vert .\Vert _1$ and $\Vert .\Vert _2$ equivalent norms in $X$ such that $(X,\Vert .\Vert _1)$ and $(X,\Vert .\Vert _2)^*$ satisfy the Mazur--Ulam property.
\end{corollary}

\medskip

\section{The Mazur--Ulam property in $C_0(K,X)$}

\medskip

Let $X$ be a Banach space, and let $K$ be  a  locally compact Hausdorff space. Given  $p\in S_{C_0(K,X)}$ and a compact open  subset $\mathcal{O}$ of $K$ such that $\Vert p(t)\Vert =1$ for all $ t\in \mathcal{O}$, we define  \emph{the intersection face of $p$ relatively to $\mathcal{O}$} by 	$$ \mathcal{F}(p,\mathcal{O}) := \{ z\in C_0(K,X) :  z(t) = p(t) \ \mbox{for all}\  t\in \mathcal{O}\}.$$

\begin{proposition}\label{C(K)-strictly convex} Let $X$ be a strictly convex  Banach space such that dim$(X) \geq 2 $ and the set $\it{Sm}(X)$  is norm dense in $S_{X }$, let $K$ be a  locally compact Hausdorff space with $\vert K \vert \geq 2$. Let $Y$ be a Banach space, and let $\Delta : S_{C_0(K,X)}\to S_Y$ be a surjective isometry such that $\Delta \arrowvert_{\mathcal{F}(p,\mathcal{O})}$ is an affine map where  $p \in S_{C_0(K,X)}$ and  $\mathcal{O}$ is a compact open subset of $K$ with $\Vert p(t)\Vert=1$ for every $t\in \mathcal{O}$ and $K\setminus \mathcal{O}\neq \emptyset$. Then  the equality $$\psi \Delta (p) = \varphi_{x} \otimes \delta_{t} (p) $$ is true for all  $t\in K\setminus \mathcal{O}$, $x\in \it{Sm}(X)$ and $\psi\in {\rm supp}(t,x)$.
\end{proposition}
\begin{proof}  Let us fix  $t_0\in K\setminus \mathcal{O}$, $x_0\in \it{Sm}(X)$ and $\psi\in {\rm supp}(t_0,x_0)$. For $j=1,2$, we set $p_j = \chi_{_{\mathcal{O}_j}} p$  where $\mathcal{O}_1=K\setminus \mathcal{O}$ and  $\mathcal{O}_2=\mathcal{O}$. It follows that
	$$ \mathcal{F}(p_2,\mathcal{O}_2) = \{ a\in C_0(K,X) : a(t) = p(t), \ \forall t\in \mathcal{O}_2\} = p_2 + B_{C_0(\mathcal{O}_1, X)}.$$ Since $\Delta \arrowvert_{\mathcal{F}(p,\mathcal{O}_2)}$  is affine, the mapping  $\Delta_{p_2}$ from $B_{C_0(\mathcal{O}_1, X)}$ to  $\Delta (\mathcal{F}(p,\mathcal{O}_2))$ defined by  $\Delta_{p_2}(g):=\Delta (p_2+g)$ for every  $g\in C_0(\mathcal{O}_1, X)$ is  affine. Now, the mapping $\varphi$ from $B_{C_0(\mathcal{O}_1, X)}$ to  $\mathbb{R}$ defined  by $$\varphi (g) := \psi \Delta (g+ p_2) - \psi\Delta(p_2)$$ is  affine with $\vert \varphi (g) \vert \leq 1$ for every $g\in B_{C_0(\mathcal{O}_1, X)}$. We can regard $\varphi$ as the restriction to $B_{C_0(\mathcal{O}_1, X)}$  of a  functional $\varphi '$ in  $B_{C_0(\mathcal{O}_1,X)^*}$. Let $\widetilde{\varphi}$ be the function from $C_0(K,X)$ to $\RR$ defined by $\widetilde{\varphi} (g) := \varphi '(g|_{ \mathcal{O}_1})$ for every $g\in C_0(K,X)$. It follows that $\widetilde{\varphi}$ belongs to $B_{C_0(K,X)^*}$. By Lemma \ref{l tech 2-bis} $(5)$, $\psi\Delta(p_2)=0$ since $p_2 (t_0) =0$ and $\psi\in {\rm supp}(t_0,x_0)$, and hence  $\widetilde{\varphi} (g) = \psi \Delta (g|_{\mathcal{O}_1}+ p_2)$ for all $g\in C_0(K, X)$. We set $f_0 = x_0\otimes \chi_{_{\mathcal{O}_1}}$. It is clear that $f_0+ p_2\in A(t_0,x_0)$, and since $\psi\in {\rm supp}(t_0,x_0)$, we get $\widetilde{\varphi} (f_0) = \psi\Delta (f_0 + p_2)=1,$ and so $\widetilde{\varphi}\in S_{C_0(K,X)^*}$.
	
	Since $\psi\Delta(p) = \widetilde{\varphi} (p) $, the proof of the proposition is concluded if we prove that
	$$\widetilde{\varphi} (p) = \varphi_{x_0}\otimes \delta_{t_0} (p).$$  Actually, we shall show more, namely that \begin{equation}\label{eq widetildevarphi and varphixootimesto coincide} \widetilde{\varphi} (g) = \varphi_{x_0}\otimes \delta_{t_0} (g) \ \mbox{for all} \ g\in C_0(K,X).
	\end{equation}
	Let us fix  $\varepsilon>0$ and $g\in C_0(K,X)$.
	Since  $x_0\in \it{Sm}(X)$, it follows from \cite[Theorem I.1.4]{DevGodZiz93} that there exists $0<\rho<\frac{\varepsilon}{4}$ with the following property: \begin{equation}\label{eq G-diff at x0}  |\phi (g(t_0)) -\varphi_{x_0} (g(t_0)) |< \frac{\varepsilon}{4}\ \mbox{ for every}\  \phi\in B_{X^*}\  \mbox{with}\   \phi (x_0)>1-\rho .
	\end{equation} Since  $\it{Sm}(X)$ is norm dense in $S_{X }$, the set  $$\{ \varphi_x \otimes \delta_{t} : t\in K,\ \ x\in \it{Sm}(X) \}\subseteq C_0(K,X)^*$$ is a norming set, and therefore, by the Hahn-Banach theorem, \begin{equation}\label{eq weak* convex hull} \overline{co}^{w^*} \{ \varphi_x \otimes \delta_{t} : t\in K,\ \ x\in \it{Sm}(X) \} =B_{C_0(K,X)^*}.
	\end{equation} Let $\mathcal{O}_3$ denote the open set $ \{t\in \mathcal{O}_1 : \|g(t)-g(t_0)\|< {\rho} \}.$ By Urysohn's lemma, there exists  $h_0\in C_0(K)$ such that $0\leq h_0\leq 1$, $h_0(t_0)=1$ and $h_0|_{K\backslash\mathcal{O}_3}=0$. We set $\widetilde{f}_0 = x_0\otimes h_0$. By (\ref{eq weak* convex hull}), there exist $\lambda_1,\ldots,\lambda_n\in (0,1]$ with $\sum_{i=1}^{n} \lambda_i$, $t_1,\ldots,t_n\in K$, and $x_1,\ldots, x_n\in \it{Sm}(X)$ such that
	\begin{equation}\label{eq (2)} \left|\widetilde{\varphi} (v) - \! \left(\!\sum_{i=1}^{n} \lambda_i\ \varphi_{x_i}\otimes \delta_{t_i} \!\right)\! (v) \right| <{\rho}, \ \ v\in \{ \widetilde{f}_0,g\} .
	\end{equation}
	It is clear that $\widetilde{f}_{0}\in A(t_0,x_0)$, and  since $\widetilde{f}_{0}|_{\mathcal{O}_2}=0$, we get $\widetilde{f}_{0}+p_2\in A(t_0,x_0)$, and so  $$\widetilde{\varphi} (\widetilde{f}_0) =\psi\Delta (\widetilde{f}_0 + p_2)=1.$$
Now we follow  the argument in the proof of the Proposition \ref{c-0-suma-estric-convex-norming}, and we check that (\ref{eq widetildevarphi and varphixootimesto coincide}) is true.
\end{proof}

\begin{proposition}\label{Afin-Face} Let $X$ be a strictly convex  Banach space such that dim$(X) \geq 2 $, let $K$ be a totally disconnected locally compact Hausdorff space with $\vert K \vert \geq 2$. Let $Y$ be a Banach space, and let $\Delta : S_{C_0(K,X)}\to S_Y$ be a surjective isometry. Then the restriction of	$\Delta$ to $\mathcal{F}(p,\mathcal{O})$, regarded as a mapping onto $\Delta (\mathcal{F}(p,\mathcal{O}))$, is an affine map for every $p \in S_{C_0(K,X)}$  and every   compact open subset $\mathcal{O}$ of $K$ with $\Vert p(t)\Vert =1$ for every $t \in \mathcal{O}$ and $K\setminus \mathcal{O}\neq \emptyset$.
\end{proposition}

\begin{proof} Let $p$ be in $S_{C_0(K,X)}$,   and let $\mathcal{O}$ be a compact open subset with $\Vert p(t)\Vert =1$ for every $t \in \mathcal{O}$ and $K\setminus \mathcal{O}\neq \emptyset$. By \cite[Lemma 8]{MoriOza2018}, $\Delta(\mathcal{F}(p,\mathcal{O}))$ is an intersection face, and hence a non-empty convex subset of $S_Y$, since $\displaystyle \mathcal{F}(p, \mathcal{O}) = \bigcap_{t\in \mathcal{O}} A(t, p(t))$ is a intersection face.	Fix $z_1,z_2\in \mathcal{F}(p,\mathcal{O}) $, $\alpha \in (0,1)$, and $\varepsilon >0$. There exist  $a_1, a_2\in C_{0}(K\setminus \mathcal{O},X)$ with $\Vert (1-\chi_\mathcal{O})z_i-a_i\Vert <\varepsilon$  and $\overline{\{t\in K: a_i(t)\neq 0 \}} $ is a non-empty compact subset of $K\setminus \mathcal{O}$ ($i=1,2$). For $i=1,2$, we define $z_i'(t):= z_i(t)$ if $t\in \mathcal{O}  $ and $z_i'(t):=a_i(t)$ if $t\in K\setminus \mathcal{O}$. Then $\Vert z_i-z_i'\Vert <\varepsilon$ and $\overline{\{t\in K: z_i'(t)\neq 0 \}}$ is a compact subset of $K$ ($i=1,2$). Since $K$  has a basis of its  topology consisting of compact open sets, there exist $\mathcal{O}_1,\cdots ,\mathcal{O}_n$  compact open subsets of $K$ with $$\mathcal{O} \varsubsetneq\cup_{i=1,2}\overline{\{t\in K: z_i'(t)\neq 0 \}}\subseteq \cup_{i=1}^n\mathcal{O}_i.$$ Now, $K_1:=\cup_{i=1}^n\mathcal{O}_i$ is a totally disconnected  compact Hausdorff subset of $K$, $z_1',z_2'\in C(K_1,X)$,  and  $C_0(K,X)=C(K_1,X)\oplus_\infty C_0(K_2,X)$. We set 	$$ \mathcal{F}_{K_1}(p,\mathcal{O}) := \{ a\in C(K_1,X) : a(t) = p(t), \ \mbox{for all}\  t\in \mathcal{O}\}.$$ It is clear that $z_1',z_2'\in \mathcal{F}_{K_1}(p,\mathcal{O}) \subseteq \mathcal{F}(p,\mathcal{O})$ and that $\mathcal{F}_{K_1}(p,\mathcal{O})=\chi_\mathcal{O}p+B_{C(K_1\setminus \mathcal{O},X)}$.
	
Let us  verify that the set $\Delta (\mathcal{F}_{K_1}(p,\mathcal{O}))$ is a convex subset of $S_Y$. 	
Let $a,b\in \Delta (\mathcal{F}_{K_1}(p,\mathcal{O}))\subseteq \Delta (\mathcal{F} (p,\mathcal{O}))$ and $\beta \in (0,1)$. Then $\beta a+(1-\beta )b\in \Delta (\mathcal{F}(p,\mathcal{O}))$. Take $u,v\in \mathcal{F}_{K_1}(p,\mathcal{O})$ and $w\in \mathcal{F} (p,\mathcal{O})$ such that $\Delta (u)=a$, $\Delta (v)=b$ and $\Delta (w)=\beta a+(1-\beta )b$. By Lemma \ref{l tech 2-bis} (5), for all $t \in K \setminus K_1$ and $x\in S_{X}$, we have $\psi \Delta(u) =0 =\psi \Delta(v)$ for every $\psi \in {\rm supp}(t ,x)$. This implies that $\psi \Delta(w)=0$ for every $\psi \in {\rm supp}(t ,x)$. By Lemma \ref{l tech 2-bis}(6), $w(t)=0$ for  $t \in K \setminus K_1$. It follows that $w\in  \mathcal{F}_{K_1}(p,\mathcal{O})$, and hence $\Delta (\mathcal{F}_{K_1}(p,\mathcal{O}))$ is a convex subset of $S_Y$, as desired. Furthermore, the mapping $\Delta_{\chi_\mathcal{O}p} : B_{C(K_1\setminus \mathcal{O},X)}\to \Delta (\mathcal{F}_{K_1}(p,\mathcal{O}))$ defined by $\Delta_{\chi_{\mathcal{O}} p}(z):=\Delta (\chi_\mathcal{O}p+z)$ is a surjective isometry.  Therefore, by  Proposition \ref{Strong Mankiewicz property} ($K_1\setminus \mathcal{O}$ is totally disconnected  compact Hausdorff space) and Theorem \ref{Mori-Ozawa},  $\Delta_{\chi_\mathcal{O}p}$ is an affine map, and so the restriction of	$\Delta$ to $\mathcal{F}_{K_1}(p,\mathcal{O})$, regarded as a mapping onto $\Delta (\mathcal{F}_{K_1}(p,\mathcal{O}))$, is affine. This fact, together with  the inequalities $$\Vert \Delta (\alpha z_1+(1-\alpha )z_2)- \Delta (\alpha z'_1+(1-\alpha )z'_2)\Vert <\varepsilon $$ and $$\Vert \alpha \Delta (z_1)+(1-\alpha)\Delta (z_2) -\alpha \Delta (z'_1)-(1-\alpha)\Delta (z'_2)\Vert <\varepsilon ,$$ allows us to obtain $$\Vert \Delta (\alpha z_1+(1-\alpha )z_2)-\alpha \Delta (z_1)-(1-\alpha)\Delta (z_2)  \Vert <2\varepsilon .$$ 	By letting $\varepsilon \to 0$, we realize that  the restriction of	$\Delta$ to $\mathcal{F}_{K_1}(p,\mathcal{O})$, regarded as a mapping onto $\Delta (\mathcal{F}_{K_1}(p,\mathcal{O}))$, is affine.\end{proof}

\begin{proposition}\label{C(K)-disconex} Let $X$ be a strictly convex  Banach space such that dim$(X) \geq 2 $ and such that the set $\it{Sm}(X)$  is norm dense in $S_{X }$, and let $K$ be a totally disconnected  compact Hausdorff space with $\vert K \vert \geq 2$. Then $C(K,X)$ satisfies the Mazur--Ulam property.
\end{proposition}
\begin{proof}   We recall  that  the subset  $\{ \varphi_x \otimes \delta_{t} : t\in K,\ \ x\in \it{Sm}(X) \}$ of $ C(K,X)^*$ is a norming set. Now, keeping in mind the proof of Theorem \ref{suma-Mazur-Ulam}, the proof shall be concluded by showing that the equality $$\psi \Delta (g) = \varphi_{x} \otimes \delta_t (g) $$ holds for all $\psi\in {\rm supp}(t ,x)$, $t\in K,\ \ x\in \it{Sm}(X)$,  and $g\in S_{C(K,X)}$. Let $\varepsilon >0$, and take $g\in S_{C(K,X)}$ and $t_0\in K$ with $\Vert g(t_0)\Vert =1$. There exists a clopen subset $\mathcal{O}_1$ of $K$ such that $$t_0\in \mathcal{O}_1\subseteq \{ s\in K :\Vert g(s)-g(t_0)\Vert <\frac{\varepsilon }{2} \} \ \mbox{and}\ K\setminus \mathcal{O}_1\neq \emptyset.$$ Define $h\in S_{C(K,X)}$ by $h:=\chi_{_{\mathcal{O}_1}}g(t_0)+\chi_{_{\mathcal{O}_2}}g$ where $\mathcal{O}_2=K\setminus \mathcal{O}_1$. Then  $\Vert h-g\Vert <\varepsilon$. By Proposition \ref{C(K,X)-Strong Mankiewicz property}, there exist $\lambda_1,\cdots,\lambda_n$ in $(0,1]$ with $\sum_{i=1}^{n} \lambda_i=1$, and $p_1,\cdots,p_n$ in $Ext(B_{C(K,X)})$  such that \linebreak $h=\sum_{i=1}^{n} \lambda_i p_i$. Therefore, $g(t_0)=\sum_{i=1}^{n} \lambda_i p_i(t)$  for all $t\in \mathcal{O}_1$, and hence $g(t_0)=p_i(t)$  for all $t\in \mathcal{O}_1$ and $i\in \{ 1, \ldots ,n\}$, since $g(t_0)$ is a extreme point of $B_X$. For each $i\in \{ 1, \ldots ,n\}$, we denote $p_i^j:=p_i\rvert_{\mathcal{O}_j}$ ($j=1,2$). Then $$h=\sum_{i=1}^{n} \lambda_i (\chi_{_{\mathcal{O}_1}}g(t_0)+p_i^2)=\chi_{_{\mathcal{O}_1}}g(t_0) +\sum_{i=1}^{n} \lambda_i p_i^2.$$ We recall that $$ \mathcal{F}(h, \mathcal{O}_1)  = \chi_{\mathcal{O}_1}g(t_0) + B_{C(\mathcal{O}_2, X)}$$  is an intersection face,  by \cite[Lemma 8]{MoriOza2018}, $\Delta(\mathcal{F}(h, \mathcal{O}_1 )$ is  a non-empty convex subset of $S_Y$. Furthermore, the mapping $$\Delta_{h} : B_{C(\mathcal{O}_2, X)}\to \Delta(\mathcal{F}(h,\mathcal{O}_1))$$ defined by $\Delta_{h}(z):=\Delta (\chi_{_{\mathcal{O}_1}}g(t_0)+z)$ for $z\in C(\mathcal{O}_2, X)$ is a surjective isometry. By  Proposition \ref{C(K,X)-Strong Mankiewicz property} and Theorem \ref{Mori-Ozawa},  $\Delta_{h}$ is  affine. Then $\Delta|_{\mathcal{F}(h,\mathcal{O}_1)}$ is an affine map, and so  $\Delta (h)=\sum_{i=1}^{n} \lambda_i \Delta (\chi_{\mathcal{O}_1}g(t_0)+p_i^2)$. For all $x\in \it{Sm}(X)$, $t\in K$, and $\psi\in {\rm supp}(t ,x)$, we have that $\psi \Delta (h)=\sum_{i=1}^{n} \lambda_i \psi \Delta (\chi_{\mathcal{O}_1}g(t_0)+p_i^2)$.  Now, note that $\Vert (\chi_{\mathcal{O}_1}g(t_0)+p_i^2)(t)\Vert =1$ for all $t\in K$ and $i\in \{ 1, \cdots ,n\}$, and hence, by Proposition \ref{C(K)-strictly convex},  for all $x\in \it{Sm}(X)$, $t\in K$, and $\psi\in {\rm supp}(t ,x)$, it follows $$\psi \Delta (\chi_{\mathcal{O}_1}g(t_0)+p_i^2) = \varphi_{x} \otimes \delta_t (\chi_{\mathcal{O}_1}g(t_0)+p_i^2) ,$$ and  hence
$$\psi \Delta (h) = \varphi_{x} \otimes \delta_t (h). $$ By letting $\varepsilon \to 0$, we realize that  the equality $$\psi \Delta (g) = \varphi_{x} \otimes \delta_t (g) $$ is true for all $x\in \it{Sm}(X)$, $t\in K$,  $\psi\in {\rm supp}(t ,x)$, and $g\in S_{C(K,X)}$.\end{proof}

We can now establish the main result of this section.

\begin{theorem}\label{C(K)-disconex-non-smooth-localy}  Let $X$ be a strictly convex  Banach space such that dim$(X) \geq 2 $ and such that the set $\it{Sm}(X)$  is norm dense in $S_{X }$, and let $K$ be a totally disconnected  locally compact Hausdorff space with $\vert K \vert \geq 2$. Then $C_0(K,X)$ satisfies the Mazur--Ulam property.
\end{theorem}
\begin{proof}  We can assume that $K$ is not a compact space, since in other case we can apply Proposition \ref{C(K)-disconex-non-smooth}.	  Let  $Y$ be a Banach space, and let \linebreak $\Delta : S_{C_0(K,X)}\to S_Y$ be a surjective isometry.  We recall  that  the subset  $\{ \varphi_x \otimes \delta_{t} : t\in K,\ \ x\in \it{Sm}(X) \}$ of $C_0(K,X)^*$ is a norming set. Now, keeping in mind the proof of Theorem \ref{suma-Mazur-Ulam}, the proof shall be concluded by showing that the equality $$\psi \Delta (g) = \varphi_{x} \otimes \delta_t (g) $$ is true for all $x\in \it{Sm}(X)$, $t\in K$,  $\psi\in {\rm supp}(t ,x)$, and $g\in S_{C(K,X)}$. Let $0<\varepsilon <\frac{1}{3}$, and take $g\in S_{C_0(K,X)}$. There exists $t_0\in K$ with $\Vert g(t_0)\Vert =1$. Since $K$ is a totally disconnected locally compact Hausdorff  space, there exists a compact open  subset $\mathcal{O}_1$ of $K$ such that $$t_0\in \mathcal{O}_1\subseteq \{ s\in K :\Vert g(s)-g(t)\Vert <\frac{\varepsilon }{2} \} \ \mbox{and}\  K\setminus \mathcal{O}_1\neq \emptyset.$$ Define $h:=\chi_{\mathcal{O}_1}g(t_0)+\chi_{\mathcal{O}_2}g$ where $\mathcal{O}_2=K\setminus \mathcal{O}_1$. We have that $\Vert h-g\Vert <\varepsilon$ and that $h\in S_{C_0(K,X)}$. By Proposition \ref{Afin-Face}, the restriction of	$\Delta$ to $\mathcal{F}(h,\mathcal{O}_1)$, regarded as a mapping onto $\Delta (\mathcal{F}(h,\mathcal{O}_1))$, is an affine map. Since $h\in S_{C_0(K,X)}$ and $K$ is not a compact space,  pick a compact open subset $\mathcal{O}_3$ of $K$ with  $\Vert h(t)\Vert <\varepsilon$ for every $t\in \mathcal{O}_3$ and  $x_3\in S_X$. For $i=1,2$, define the functions $p_i:=\chi_{\mathcal{O}_3} (-1)^ix_3+\chi_{\mathcal{O}_4}h$ where $\mathcal{O}_4=K\setminus \mathcal{O}_3$. Then $\mathcal{O}_3\cap \mathcal{O}_1=\emptyset$, $p_1,p_2\in \mathcal{F}(h,\mathcal{O}_1)$ and $\Vert h-\frac{1}{2}(p_1+p_2)\Vert <\varepsilon$. By Proposition \ref{C(K)-strictly convex}, $$\psi \Delta (p_i) = \varphi_{x} \otimes \delta_t (p_i) $$  for all  $x\in \it{Sm}(X)$, $t\in K$, and $\psi\in {\rm supp}(t ,x)$. Since $\Delta|_{\mathcal{F}(h,\mathcal{O}_1)}$ is an affine map, we have that $$\vert \psi \Delta (h)-\varphi_{x} \otimes \delta_t (h)\vert $$ $$ =\vert \psi \Delta (h)-\frac{1}{2}(\psi \Delta (p_1)+\psi \Delta (p_2))+ \varphi_{x} \otimes \delta_t (\frac{1}{2}(p_1+p_2))  -\varphi_{x} \otimes \delta_t (h) \vert$$
	
$$\leq \vert \psi \Delta (h)-\psi \Delta (\frac{1}{2} (p_1+p_2))\vert + \vert \varphi_{x} \otimes \delta_t (\frac{1}{2}(p_1+p_2))  -\varphi_{x} \otimes \delta_t (h) \vert $$

$$\leq \Vert \Delta (h)- \Delta (\frac{1}{2} (p_1+p_2)) \Vert +\Vert h- \frac{1}{2}(p_1+p_2)\Vert $$ $$=2 \Vert h- \frac{1}{2}(p_1+p_2)\Vert <2\varepsilon . $$

Since $\Vert g-h\Vert <\varepsilon$, $$\vert \psi \Delta (h)-\psi \Delta (g)\vert <\varepsilon$$ and $$\vert \varphi_{x} \otimes \delta_t (h)-\varphi_{x} \otimes \delta_t (g)\vert <\varepsilon .$$ It follows that $\vert \psi \Delta (h)-\varphi_{x} \otimes \delta_t (h)\vert <4\varepsilon .$ 	By letting $\varepsilon \to 0$, we obtain that   $$\psi \Delta (g) = \varphi_{x_0} \otimes P_{\gamma_0} (g)$$ for all  $x\in \it{Sm}(X)$, $t\in K$, and  $\psi\in {\rm supp}(t ,x)$.\end{proof}

We conclude this paper  obtaining a refinement of the previous theorem for a wide range of  locally  compact Hausdorff  spaces. By \cite[Theorem 3.44, p. 92]{AliHit}, the one-point compactification $K^*$ of a noncompact locally compact Hausdorff space $K$ is metrizable if and only if $K$ is second countable. By \cite[Lemma 3.99, p. 125]{AliHit}, $C(K^*)$ is separable if and only if $K^*$ is metrizable. It follows that  a locally compact Hausdorff space $K$ is second countable if and only if $C_0(K)$ is a separable Banach space.

\medskip

\begin{proposition}\label{C(K)-disconex-non-smooth} Let $X$ be a strictly convex  Banach space such that dim$(X) \geq 2 $, and let $K$ be a totally disconnected locally compact Hausdorff space such that $\vert K \vert \geq 2$ and such that $K$ is second countable. Then $C_0(K,X)$ satisfies the Mazur--Ulam property.
\end{proposition}
\begin{proof}  Let $Y$ be a Banach space, and let $\Delta : S_{C_0(K,X)}\to S_Y$ be a surjective isometry. Let $A$ be a separable subset of $S_{C_0(K,X)}$.
First, we check that  there exists a separable closed subspace $M$ of $X$ such that $$A\subseteq S_{C_0(K,M)}\subseteq S_{C_0(K,X)} .$$ Fix  a countable subset $\{ a_k:k\in \NN \ \  a_k\in A \}$  which is dense in $A$. Take $k,n\in \NN$. There exists  $a_k'\in C_{0}(K,X)$  such that $\Vert a_k-a_k'\Vert <\frac{1}{n}$  and $\overline{\{t\in K: a_k'(t)\neq 0 \}} $ is a compact subset of $K$. There exist $\mathcal{O}_1,\cdots ,\mathcal{O}_m$  compact open subsets of $K$ with $\overline{\{t\in K: a_k'(t)\neq 0 \}}\subseteq \cup_{i=1}^m\mathcal{O}_i$. Then we have $C_0(K,X)=C(K_1,X)\oplus_\infty C_0(K_2,X)$ where $K_1:=\cup_{i=1}^m\mathcal{O}_i$ and $K_2:=K\setminus K_1$. It is clear that $a_k'\in C(K_1,X)$ and $K_1$ is a totally disconnected  compact Hausdorff space.
For each $t\in K_1$, define the open subset of $K_1$ $$V_t:=\{ s\in K_1: \Vert a_k'(s)-a_k'(t) \Vert <\frac{1}{2n}    \}.$$ There exists  a clopen subset $\mathcal{O}_t$ of $K_1$ such that $t\in \mathcal{O}_t\subseteq V_t$ for every $t\in K_1$. Since $K_1$ is compact and $K_1=\cup_{t\in K_1}	\mathcal{O}_t$,  there exists a finite subset $C_{(n,k)}$ of $K_1$ such that $K_1= \cup_{t\in C_{(n,k)}}	\mathcal{O}_t$. We can assume that $\mathcal{O}_t\cap \mathcal{O}_{t'}=\emptyset$ whenever $t\neq t'$ and $t,t'\in C_{(n,k)}$.  We define the function $h_{(n,k)} := \sum_{t\in C_{(n,k)}} \chi_{O_t} a_k'(t)$.  For $s\in K_1$, there exists $t\in C_{(n,k)}$ such that $s\in \mathcal{O}_t$, and hence $$\Vert a_k'(s)-h_{(n,k)}(s)\Vert =\Vert a_k'(s)- a_k'(t)\Vert <\frac{1}{n}.$$Then $\Vert a_k'-h_{(n,k)}\Vert _{C(K_1,X)} =\Vert a_k'-h_{(n,k)}\Vert\leq \frac{1}{n}$. This implies that  $\Vert a_k-h_{(n,k)}\Vert\leq \frac{2}{n}$.
We consider the separable closed subspace  $$M_0=\overline{Lin}   \{a_k'(t):k,n\in \NN ,\  t\in C_{(n,k)} \} .$$ It follows that $C_0(K,M_0)$ is a closed separable subspace of $C_0(K,X)$, and that  $h_{(n,k)}\in C_0(K,M_0)$ for all $n,k\in \NN$, and hence $A\subseteq S_{C_0(K,M_0)}$.
We follow the argument in the proof of Lemma \ref{separable} to derive the existence of separable closed subspaces $M\subseteq X$ and $N\subseteq Y$ such that $A\subseteq S_{C_0(K,M)}$ and  $\Delta ( S_{C_0(K,M)})=S_N$.
Let us take $z,\tilde{z}\in S_{C_0(K,X)}$ and $\lambda >0$. By the previous paragraph, there exist separable closed subspaces $M\subseteq X$ and $N\subseteq Y$ such that $z, \tilde{z}\subseteq S_{C_0(K,M)}$ and $\Delta ( S_{C_0(K,M)})=S_N$. Since $X$ is a strictly convex,  $M$ is a strictly convex and  by  Mazur's theorem   \cite[Proposition 9.4.3]{Ro}), the set of all smooth points of $M$ is dense in $S_M$.  By Proposition \ref{C(K)-disconex}, the restriction of	$\Delta$ to $S_{C_0(K,M)}$, regarded as a mapping onto $S_N$, admits an extension to a surjective real-linear isometry $\Phi$ from $C_0(K,M)$ to $N$. It follows that $$\Vert \Delta (z) +\lambda \Delta (\tilde{z})\Vert = \Vert \Phi (z +\lambda \tilde{z})\Vert =\Vert z+\lambda \tilde{z} \Vert ,$$
and hence by \cite[Lemma 2.1]{FangWang06},  $\Delta$ can be extended to  a real-linear isometry of $C_0(K,X)$ onto $Y$.
\end{proof}

A 1910 theorem of Brouwer characterizes the Cantor set as the unique totally disconnected, compact metric space without isolated points.

\begin{corollary}\label{Cantor}  Let $X$ be a strictly convex  Banach space such that dim$(X) \geq 2 $, and let $ \mathfrak{C}$ be the Cantor set. Then $C(\mathfrak{C},X)$ satisfies the Mazur--Ulam property.
\end{corollary}

\section*{Acknowledgements}

The  author is very grateful to M. Cabrera and A. Rodr\'iguez Palacios for fruitful remarks concerning the matter of the paper. The author is  partially supported  by the Junta de Andaluc\'ia grant FQM199.

\end{document}